\documentclass{amsart}

\usepackage{amssymb}
\usepackage[dvips]{epsfig}
\usepackage{graphicx}
\usepackage{color}
\newcommand{\abs}[1]{\lvert#1\rvert}

\newcommand{\C}{\mathbb{C}}
\newcommand{\R}{{\mathbb R}}
\newcommand{\N}{{\mathbb N}}

\newcommand{\eps}{\varepsilon}

\newcommand{\cZ}{{\mathcal Z}[\zeta]}

\newcommand{\p}{\mathfrak{p}}

\usepackage{amsmath,amssymb,amsthm}
\usepackage{hyperref}

\def\dis
{\displaystyle}

\def\C{{\mathbb C}}
\def\R{{\mathbb R}}
\def\N{{\mathbb N}}
\def\T{{\mathbb T}}

\def\cH{\mathcal H}
\def\cF{\mathcal F}
\def\cM{\mathcal M}
\def\cO{\mathcal O}

\def\cI{\mathcal I} 
\def\cX{\mathcal X} 
\def\cU{\mathcal U} 
\def\cZ{\mathcal Z} 
\def\cA{\mathcal A} 
\def\cP{\mathcal P} 
\def\cS{\mathcal S}

\theoremstyle{plain}
\newtheorem{theorem}{Theorem}[section]
\newtheorem{lemma}[theorem]{Lemma}
\newtheorem{corollary}[theorem]{Corollary}
\newtheorem{proposition}[theorem]{Proposition}
\newtheorem{question}{Question}

\theoremstyle{definition}

\newtheorem{remark}[theorem]{Remark}
\newtheorem*{remark*}{Remark}

\def\virgp{\raise 2pt\hbox{,}}

\def\({\left(}
\def\){\right)}
\def\<{\left\langle}
\def\>{\right\rangle}
\def\le{\leqslant}
\def\ge{\geqslant}

\newcommand{\Int}{\displaystyle \int}
\newcommand{\Frac}{\displaystyle \frac}
\renewcommand{\div}{{\rm div}\,}

\def\Eq#1#2{\mathop{\sim}\limits_{#1\rightarrow#2}}
\def\Tend#1#2{\mathop{\longrightarrow}\limits_{#1\rightarrow#2}}

\def\d{{\partial}}

\def\eps{\varepsilon}

\def\si{{\sigma}}

\DeclareMathOperator{\RE}{Re}
\DeclareMathOperator{\IM}{Im}
\DeclareMathOperator{\DIV}{\mathrm{div}}

\numberwithin{equation}{section}

\begin{document}
\title[Madelung, Gross--Pitaevskii and Korteweg]{Madelung,
  Gross--Pitaevskii and Korteweg} 
\author[R. Carles]{R\' emi Carles}
\address{CNRS \& Univ. Montpellier~2\\Math\'ematiques, UMR 5149
\\CC~051\\34095
  Montpellier\\ France}
\email{Remi.Carles@math.cnrs.fr}

\author[R. Danchin]{Rapha\"{e}l Danchin}
\address{LAMA UMR CNRS 8050,
Universit\'e Paris EST\\
61, avenue du G\'en\'eral de Gaulle\\
94010 Cr\'eteil Cedex\\ France}
\email{danchin@univ-paris12.fr}
\author[J.-C. Saut]{Jean-Claude Saut}
\address{Laboratoire de Math\' ematiques, UMR 8628\\
Universit\' e Paris-Sud et CNRS\\ 91405 Orsay, France}
\email{jean-claude.saut@math.u-psud.fr}

\begin{abstract}
This paper surveys various aspects of the hydrodynamic formulation of
the nonlinear Schr\"{o}dinger equation obtained via the Madelung
transform in connexion to models of quantum hydrodynamics and to
compressible fluids of the Korteweg type. 
\end{abstract}
\thanks{This work was supported by the French ANR projects Equa-disp
(ANR-07-BLAN-0250-01) and 
  R.A.S. (ANR-08-JCJC-0124-01).}

\maketitle

\section{Introduction}

In his seminal work \cite{Madelung} (see also \cite{LandauQ}), E. Madelung introduced 
the so-called {\it Madelung transform} in order to
relate the linear Schr\"{o}dinger equation to  a {\it hydrodynamic
  type system}. This system takes (slightly) different forms according
to the context: linear or nonlinear equation with various
nonlinearities. 
For the semi-classical  nonlinear Schr\"{o}dinger equation (shortened
in  NLS in what follows): 
\begin{equation*}
  i\eps\d_t \psi^\eps +\frac{\eps^2}{2}\Delta \psi^\eps = f\(\lvert
  \psi^\eps\rvert^2\)\psi^\eps, 
\end{equation*}
the Madelung transform  amounts to setting
$$ \psi^\eps(t,x)=
  \sqrt{\rho(t,x)} e^{i\phi(t,x)/\eps},$$ 
so as to get  the following  system for  $\rho$ and $v=\nabla \phi$: 
\begin{equation}
  \label{eq:qhd}
  \left\{
    \begin{aligned}
      &\d_t v +v\cdot \nabla v +
      \nabla f\(\rho\)= \frac{\eps^2}{2}\nabla\( \frac{\Delta\(\sqrt
      \rho\)}{\sqrt \rho}\) ,\\
& \d_t \rho  + \DIV \(\rho v\) =0.  
    \end{aligned}
\right.
\end{equation}

This system is referred to as the  \emph{hydrodynamic form} of NLS because
of its similarity with the compressible Euler equation (which corresponds to $\eps$=0). 
The additional  term on the right-hand side is the so-called  \emph{quantum pressure}.
\smallbreak
Madelung transform is crucial to
investigate qualitative properties of the  nonlinear 
Schr\"{o}dinger equation with nonzero boundary conditions at infinity
whenever the solution is  not expected  to vanish ``too often''. 
 Of particular interest is the so-called 
 \emph{Gross--Pitaevskii equation} 
 $$ 
  i\eps\d_t \psi^\eps +\frac{\eps^2}{2}\Delta \psi^\eps = (|\psi^\eps|^2-1)\psi^\eps, 
 $$
  which corresponds to   
  $f(r)=r-1$ or, more generally, the case where 
   $f$ is a smooth function vanishing at some $r_0$ and such that $f'(r_0)<0.$
   This covers in particular the ``cubic-quintic'' NLS 
   ($f(r)=-\alpha_1+\alpha_3r-\alpha_5r^2$ with $\alpha_1,\alpha_3,\alpha_5>0$). 
   
   Using Madelung transform is also rather
popular to  study the semi-classical limit of NLS. This requires $\rho$ to be 
nonvanishing, though. We shall see  below to what extent  the presence of vacuum (that is of points where 
$\rho$ vanishes) is merely a technical problem due
to the approach related to Madelung transform. 
\smallbreak

The hydrodynamic form of NLS may be seen as a
particular case of the system of {\it quantum fluids} (QHD) 
with a suitable choice of the pressure law, which, 
in turn, enters in the class of \emph{Korteweg (or capillary) fluids}.
We aim here at further investigating the link between these  three
{\it a priori} disjoint domains, through various uses of the Madelung
transform.
 In passing, we will 
also review some more or less known facts pertaining to the
theories of quantum fluids and Korteweg fluids. The most original part
of this work is  a simple proof of a recent result of Antonelli and
Marcati \cite{AnMa,AM09}, on the global existence of weak
solutions to a quantum fluids system.

\subsection{Organization of the paper}
The paper is organized as follows.
The second  section is a review of the connexions of the Madelung
transform with  the semi-classical limit of the nonlinear  Schr\"{o}dinger
equation. In passing we briefly present  the state-of-the-art for 
the Cauchy problem for the Gross-Pitaevskii equation. 
In Section~\ref{sec:qhddirect}, we solve the quantum hydrodynamical
system \eqref{eq:qhd} by a direct method based on the use of an extended 
formulation, and explain how it may be adapted to tackle general
Korteweg fluids.  
In Section~\ref{s:JC},  we 
review the use of the Madelung  transform  and of the hydrodynamical form of the
Gross--Pitaevskii equation  to study the existence and properties
of its traveling wave solutions and of its transonic limit, both in
the steady and unsteady cases. 
We give in Section~\ref{sec:global} a simple proof of the
aforementioned result of Antonelli and Marcati (\cite{AnMa,AM09}). 
Lastly we list  in Appendix the basic  conservation laws for the
Schr\"{o}dinger, the QHD and the compressible Euler equations, 
and show that some of these laws naturally carry over to  general Korteweg fluids.

\subsection{Notations}
\begin{itemize}
\item We denote by $\abs{\cdot}_p$ ($1\le p\le\infty$) the standard
  norm of the Lebesgue spaces $L^p(\R^d)$. 
\item The standard $H^s(\R^d)$ Sobolev norm will be denoted
  $\|\cdot\|_s$.
\item We use the Fourier multiplier notation: $f(D)u$ is defined as
  ${\mathcal F}(f(D)u)(\xi)=f(\xi) 
\widehat{u}(\xi)$, where ${\mathcal F}$ and $\widehat{\cdot}$ stand
for the Fourier transform. 
\item The operator
$\Lambda=(1-\Delta)^{1/2}$ is equivalently defined using the Fourier
multiplier notation to be $\Lambda=(1+\vert D\vert^2)^{1/2}$.
\item The partial derivatives will be denoted with a subscript, {\it
    e.g.} $u_t, u_x,\dots$ or $\partial_t u, \partial_x u,\dots$
    or even $\d_j$ (to designate $\d_{x_j}$).
    
\item $C$ will denote various nonnegative absolute constants, the meaning of which will be clear
from the context.
\end{itemize}


\section{Madelung transform and the semi-classical limit of NLS}\label{s:remi} 
As already mentioned in the Introduction, we consider the equation
\begin{equation}\label{eq:nlssemi}
  i\eps\d_t \psi^\eps +\frac{\eps^2}{2}\Delta \psi^\eps = f\(\lvert
  \psi^\eps\rvert^2\)\psi^\eps\quad ;\quad \psi^\eps(0,x)=
  \sqrt{\rho_0(x)} e^{i\phi_0(x)/\eps}.
\end{equation}
Here we have in mind the limit $\eps\to 0$. 
The space variable $x$ belongs to $\R^d$ in this paragraph. The periodic
case $x\in \T^d$ being of particular interest for numerical
simulations however (see e.g. \cite{BJM2,DGM} and references therein), we will
state some analogous results in that case, 
which turn out to be a little simpler than in the case $x\in\R^d$. 
In this section, we shall focus on  two types of nonlinearity:
\begin{itemize} 
\item Cubic, defocusing nonlinearity\footnote{We could consider more
  general defocusing nonlinearities such as  $f\(\lvert
  \psi^\eps\rvert^2\)\psi^\eps =\lvert
  \psi^\eps\rvert^{2\si}\psi^\eps$, $\si\in \N$. We consider the
  exactly cubic case for the simplicity of the exposition only.}: 
$$f\(\lvert
  \psi^\eps\rvert^2\)\psi^\eps =\lvert
  \psi^\eps\rvert^{2}\psi^\eps.$$
The Hamiltonian associated to \eqref{eq:nlssemi} then reads
\begin{equation*}
  \cH^\eps_{\rm NLS}\(\psi^\eps\) = \lVert \eps \nabla
  \psi^\eps\rVert_{L^2}^2 + \lVert
  \psi^\eps\rVert_{L^{4}}^{4}. 
\end{equation*}
It is well defined on the Sobolev space $H^1(\R^d)$
in dimension  $d\le 4$.
\item  Gross--Pitaevskii equation:  
$$f\(\lvert
  \psi^\eps\rvert^2\)\psi^\eps =\(\lvert
  \psi^\eps\rvert^{2}-1\)\psi^\eps.$$ 
In this case, the natural Hamiltonian associated to \eqref{eq:nlssemi}
is 
\begin{equation*}
  \cH^\eps_{\rm GP}\(\psi^\eps\) = \lVert \eps \nabla
  \psi^\eps\rVert_{L^2}^2 + \left\lVert
  \lvert\psi^\eps\rvert^2-1 \right\rVert_{L^{2}}^{2}. 
\end{equation*}
\end{itemize}

In both cases, the Hamiltonian defines an energy space, in which existence and
uniqueness for the Cauchy problem \eqref{eq:nlssemi} have been
established.  

The case of a defocusing cubic nonlinearity is now well understood.
In dimension $d\le 4$  the corresponding NLS equation is globally 
well-posed in $H^1(\R^d),$ and the additional $H^s(\R^d)$ regularity
($s\ge 1$) is  propagated (see  the textbooks \cite{CazCourant,LP,Tao}
and the references therein). 
\medbreak
The situation is more complicated  for the Gross--Pitaevskii equation, where
the finite energy solutions $\psi^\eps$ cannot be expected to be in
$L^2(\R^d)$, since 
$|\psi^\eps|^2-1\in L^2(\R^d)$. As noticed in \cite{Zhidkov,Zhbook},
and extended in 
\cite{Gallo}, a convenient space to study the Gross--Pitaevskii equation is the
Zhidkov space:
\begin{equation*}
  X^s(\R^d)= \{  \psi \in L^\infty(\R^d)\ ;\ \nabla \psi \in
  H^{s-1}(\R^d)\}\quad\hbox{with } \ s>d/2.
\end{equation*}
 \begin{remark}\label{rem:infini}
In the case $x\in \T^d$,  the spaces $H^s(\T^d)$ and $X^s(\T^d)$ (with
obvious definitions) are the same.

In the case $x\in\R^d,$ as a consequence of the
Hardy--Littlewood--Sobolev inequality 
(see e.g.  \cite[Th.~4.5.9]{Hormander1} or
\cite[Lemma~7]{PG05}), one may show that 
if $d\ge 2$ and $\psi\in {\mathcal D}'(\R^d)$ is such that $\nabla
\psi\in L^p(\R^d)$ for some $p\in ]1,d[$, then there 
exists a constant $\gamma$ such that $\psi-\gamma \in
L^{q}(\R^d)$, with $1/p=1/q+1/d$. Morally, $\gamma$ is the limit
of $\psi$ at infinity.  If $d\ge 3$ then we can take $p=2$,
so every function in $X^s(\R^d)$ satisfies the above
property; for $d=2$, the above assumption requires a little more decay
on $\nabla \psi$ 
than general functions in $X^s(\R^2)$. 
\end{remark}

The well-posedness issue in the natural energy space 
$$E(\R^d)=\lbrace \psi\in H^1_{\text{loc}}(\R^d),\;\nabla \psi \in L^2(\R^d),\;|\psi|^2-1\in L^2(\R^d) \rbrace$$
associated to the Gross-Pitaevskii equation has been investigated only very
recently by C. Gallo in \cite{Gallo08} and P. G\'erard in
\cite{PG05,Ger08}, in dimension $d\le 4.$ 

Let us emphasize that in the $\R^d$ case, the energy space is no
longer a linear space (contrary to the case of zero boundary condition
a infinity where it is $H^1(\R^n)$),  
hence solving  the  Gross--Pitaevskii equation in $E(\R^d)$ is more complicated. However, if $d=3,4,$ one may show that $E(\R^d)$ coincides with 
$$\lbrace \psi=c(1+v), c\in \mathbb S^1, \psi\in H^1(\R^d),
2\RE\psi+|\psi|^2\in L^2(\R^d)\rbrace,$$ 
which allows to endow it with a structure of a metric space.  The case
$d=2$ is slightly more technical.

\subsection{Some issues related to the use of the Madelung transform}

In the semi-classical context, Madelung
transform  consists  in seeking 
\begin{equation}
  \label{eq:madelung}
  \psi^\eps(t,x) = \sqrt{\rho(t,x)} e^{i\phi(t,x)/\eps}
\end{equation}
for some $\rho\ge 0$ and real-valued function $\phi$. Plugging \eqref{eq:madelung}
into \eqref{eq:nlssemi} and separating real and imaginary values
yields:
\begin{equation}
  \label{eq:separ}
  \left\{
    \begin{aligned}
      &\sqrt \rho \( \d_t \phi +\frac{1}{2}\lvert \nabla \phi\rvert^2 +
      f(\rho)\)= \frac{\eps^2}{2}\Delta\sqrt \rho\quad ;\quad
      \phi_{\mid t=0} =\phi_0,\\
& \d_t \sqrt \rho + \nabla \phi\cdot \nabla \sqrt \rho
      +\frac{1}{2}\sqrt \rho \Delta \phi =0\quad ;\quad \rho_{\mid
      t=0}=\rho_0.  
    \end{aligned}
\right.
\end{equation}
Two comments are in order at this stage: the first equation shows that
$\phi$ depends on $\eps$ and the second equation shows that so does
$\rho$ in general. We shall underscore this fact by using the notation
$(\phi^\eps,\rho^\eps)$. Second, the equation for $\phi^\eps$ can be
simplified, \emph{provided that $\rho^\eps$ has no zero}. Introducing
the velocity $v^\eps = \nabla \phi^\eps$, \eqref{eq:separ} yields the
system of \emph{quantum hydrodynamics} \eqref{eq:qhd} presented in the
introduction.   

To study  the limit $\eps\to 0$ (the \emph{Euler limit}), it is
natural to consider the 
following compressible Euler equation:
\begin{equation}
  \label{eq:euler}
  \left\{
    \begin{aligned}
      &\d_t v +v\cdot \nabla v+
      \nabla f\(\rho\)= 0\quad ;\quad 
      v_{\mid t=0} =\nabla\phi_0,\\
& \d_t \rho  + \DIV \(\rho v\) =0\quad ;\quad \rho_{\mid
      t=0}=\rho_0.  
    \end{aligned}
\right.
\end{equation}
Note that the solution to \eqref{eq:euler} must not be expected to
remain smooth for all time, even if the initial data are smooth. In
\cite{MUK86} (see also \cite{Xin98}), it is shown that compactly
supported initial data lead to the formation of singularities in
finite time. In \cite{SiderisCMP}, the author constructs a solution
developing singularities in finite time, in the absence of vacuum. 
In \cite{Alinhac}, it is  shown that for rotationally invariant
two-dimensional data 
that are perturbation of size $\eps$ of a rest state, blow-up occurs
at time $T_\eps\sim \tau/\eps^2.$

\begin{remark}
  In the framework of this section, we have $f'=1>0$. As pointed out
  above, the case $f(r)=r^\sigma$, $\sigma\in \N$, could be considered
  as well. On the other hand, if $f'<0$   (corresponding to the
  semiclassical limit 
  for a focusing nonlinearity),  \eqref{eq:euler} becomes
  an elliptic system which may be solved locally-in-time  for analytic
  data (see \cite{PGX93,ThomannAnalytic}).  
  At the same time, in the case $d=1$ and $f'=-1$, it has been shown in
  \cite{GuyCauchy} that there are smooth initial data for which the Cauchy
problem \eqref{eq:euler}  has no solution.  In short,  working with analytic
data in this context   is not only convenient,
it is mandatory. 
\end{remark}

Let us emphasize that the hydrodynamical formulations for the cubic
NLS and the Gross--Pitaevskii equations   
are exactly the same as in both cases we have $\nabla f(\rho)=\nabla\rho.$ 
{}From this point of view, studying either of the equations
is mainly a matter of boundary   conditions  at infinity: $H^s(\R^d)$
is the appropriate 
space for the cubic NLS equation whereas $X^s(\R^d)$ is adapted to the
Gross--Pitaevskii equation.

In the sequel, $Z^s$ denotes  either $H^s(\R^d)$  or $X^s(\R^d).$ In
addition, we set 
\begin{equation*}
  Z^\infty = \bigcap_{s>d/2} Z^s.
\end{equation*}
\begin{theorem}\label{th:euler}
  Let $\rho_0,\phi_0\in C^\infty(\R^d)$ with $\sqrt \rho_0,\nabla \phi_0\in
  Z^s$ for some $s>{d/2+1}$. There exists a unique maximal solution 
   $(v,\rho)\in C([0,T_{\rm max});Z^s)$
   to \eqref{eq:euler}.
    In addition,  $T_{\rm max}$ is independent of $s>d/2+1$
  and 
    \begin{equation*}
   T_{\rm max}<+\infty\Longrightarrow  \int_0^{T_{\rm
       max}}\|(v,\sqrt\rho)(t)\|_{W^{1,\infty}}\,dt=+\infty. 
  \end{equation*}
Finally, if  $\nabla \phi_0$ and $\rho_0$ are smooth, nonzero and
compactly supported then $T_{\rm max}$ is finite. 
\end{theorem}
 We investigate the following natural questions:
\begin{question}\label{q:1}
Assume that $\rho_0(x)>0$ for all $x\in \R^d$. Can we say that
$\rho^\eps(t,x)>0$ for all $x\in \R^d$ and $t\in[0,T_{\rm max})$?  If
not, what is the maximal interval allowed 
for~$t$? 
\end{question}
We will also recall that despite
the appearance, the 
presence of vacuum (existence of zeroes of $\psi^\eps$) is merely a
technical problem: the Madelung transform ceases to make sense, but
a rigorous WKB analysis  is available, regardless of the presence of
vacuum. See \S\ref{sec:bkw}. 
\begin{question}\label{q:2}
  Let $\rho_0,\phi_0\in C^\infty(\R^d)$ with $\rho_0,\nabla \phi_0\in
  H^s(\R^d)$ for some $s>d/2+1.$ Suppose  that the solution $(v,\rho)$
  to \eqref{eq:euler}  
  satisfies $\rho(t,x)>0$ for
  $(t,x)\in [0,\tau]\times\R^d$. Can we construct a solution to
  \eqref{eq:qhd} in $C([0,\tau_*];H^s)$ (possibly with
  $0<\tau_*<\tau$)? 
\end{question}
We will see that in general, the answer for this question is no. Even
though from the answer to the first question, \eqref{eq:qhd} makes
sense \emph{formally}, the analytical properties associated to
\eqref{eq:qhd} are not as favorable as for
\eqref{eq:euler}. Typically, the right-hand side of the equation for
the quantum velocity need not belong to $L^2(\R^d)$. 
\begin{question}\label{q:3}
  Let $\rho_0,\phi_0\in C^\infty(\R^d)$ with $\rho_0,\nabla \phi_0\in
  X^s(\R^d)$ for some $s>d/2+1.$ Suppose  that the solution $(v,\rho)$ to \eqref{eq:euler} 
  satisfies $\rho(t,x)>0$ for
  $(t,x)\in [0,\tau]\times\R^d$. Can we construct a solution to
  \eqref{eq:qhd} in $C([0,\tau_*];X^s)$ (possibly with
  $0<\tau_*<\tau$)?
\end{question}


\subsection{Proof of Theorem~\ref{th:euler}}
\label{sec:euler}

We shall simply give the main ideas of the proof of
Theorem~\ref{th:euler}. Complete proofs can be found in
\cite{MUK86} for the cubic defocusing NLS equation in Sobolev spaces, and in
\cite{ACIHP} for the Gross-Pitaevskii equation in  Zhidkov spaces. 

In the framework of this paper, we have
\begin{equation*}
  \nabla f(\rho)= \nabla\rho. 
\end{equation*}
Introduce formally the  auxiliary function $a=\sqrt \rho$. This nonlinear
change of variable makes \eqref{eq:euler} hyperbolic symmetric:
\begin{equation}
  \label{eq:euler2}
  \left\{
    \begin{aligned}
      &\d_t v +v\cdot \nabla v+
      2a\nabla a= 0\quad ;\quad 
      v_{\mid t=0} =\nabla\phi_0,\\
& \d_t a  + v\cdot\nabla a+\frac{1}{2}a\nabla\cdot v =0\quad ;\quad a_{\mid
      t=0}=\sqrt\rho_0.  
    \end{aligned}
\right.
\end{equation}
This system is of the form
\begin{equation*}
  \d_t u +\sum_{j=1}^d A_j(u)\d_j u =0,\text{ where } u=\left(
    \begin{array}[c]{c}
v_1\\
\vdots\\
v_d\\
a
    \end{array}
\right)\in \R^{d+1},
\end{equation*}
and the matrices $A_j$ are symmetrized by the constant multiplier
\begin{equation*}
  S =\left( 
    \begin{array}[c]{cc}
I_d &0 \\
0& 4
    \end{array}
\right).
\end{equation*}
Standard analysis (see e.g. \cite{Taylor3}) shows that \eqref{eq:euler2} has
a unique maximal solution $(v,a)$ in $C([0,T_{\rm max});Z^s)$, provided that
$s>d/2+1$ and that, in addition,
$$
  T_{\rm max}<+\infty\Longrightarrow  \int_0^{T_{\rm max}}\|(v,a)(t)\|_{W^{1,\infty}}\,dt=+\infty.
$$

We can then define $\rho$ by the linear equation
\begin{equation*}
  \d_t \rho  + \DIV \(\rho v\) =0\quad ;\quad \rho_{\mid
      t=0}=\rho_0.
\end{equation*}
By uniqueness for this linear equation, $\rho =a^2$ ($a$ is
real-valued, so $\rho$ is non-negative), and $(v,\rho)$ solves
\eqref{eq:euler}.  
\smallbreak

We now briefly explain why compactly supported initial data lead to
the formation of singularities in finite time. The first remark is
that in this case, the solution to \eqref{eq:euler} has a finite speed
of propagation, which turns out to be zero: so long as $(v,\rho)$ is
smooth, it remains supported in the same compact as its initial
data. To see this, consider the auxiliary system \eqref{eq:euler2}:
the first equation is a Burgers' equation with source term $2a\nabla
a$; the second equation is an ordinary differential equation along the
trajectories of the particles. Define the trajectory by
\begin{equation}\label{eq:ray}
  \frac{d}{dt}x(t,y)= v\(t,x(t,y)\)\quad ;\quad x(0,y)=y. 
\end{equation}
For $0\le t<T_{\rm max}$, this is a global diffeomorphism of
$\R^d$, as shown by the equation
\begin{equation*}
  \frac{d}{dt}\nabla_y x(t,y)= \nabla v\(t,x(t,y)\)\nabla_y
  x(t,y)\quad ;\quad \nabla_y x(0,y)={\rm Id},  
\end{equation*}
and Gronwall lemma. Therefore, for a smooth function $f$,
\begin{equation*}
  \(\d_t f +v\cdot \nabla f\)(t,x(t,y))= \d_t
   \(f\(t,x(t,y)\)\),
\end{equation*}
and \eqref{eq:euler2} can be viewed as a system of ordinary
differential equations. 
\smallbreak

Once the non-propagation of the support of smooth solutions is
established, the end of the proof relies on a virial computation (like
in \cite{Z,Glassey}, see also \cite{CazCourant}). This computation shows
that global in 
time smooth solutions to \eqref{eq:euler2} are dispersive (see also
\cite{Serre97}). This is incompatible with the zero propagation speed of
smooth compactly supported solutions. Therefore, singularities have to
appear in finite time.


\subsection{A review of WKB analysis associated to (\ref{eq:nlssemi})}
\label{sec:bkw}

We consider initial data which are a little more general than in
\eqref{eq:nlssemi}, namely
\begin{equation}
  \label{eq:nlssemi2}
\psi^\eps(0,x)=   a_0^\eps(x)e^{i\varphi_0(x)/\eps}, 
\end{equation}
where the initial amplitude $a_0^\eps$ is assumed to be smooth,
\emph{complex-valued}, and possibly depending on $\eps$. Typically, we assume
that there exist $a_0, a_1\in Z^\infty$ independent of $\eps$ such
that 
\begin{equation}\label{eq:CIa}
  a_0^\eps = a_0+\eps a_1 +\cO(\eps^2) \text{ in }Z^s, \quad \forall
  s>d/2. 
\end{equation}

\subsubsection{First order approximation}
\label{sec:bkw1}

Introduce the solution to the quasilinear system
\begin{equation}\label{eq:eulerOG}
\left\{
  \begin{aligned}
      &\d_t \varphi + \frac{1}{2}\lvert \nabla \varphi \rvert^2 +
      f\(|a|^2\)=0 \quad ;\quad \varphi_{\mid t=0}=\varphi_0,\\
&\d_t a+\nabla \varphi\cdot \nabla a +\frac{1}{2}a\Delta \varphi=0\quad
      ;\quad a_{\mid t=0}=a_0. 
  \end{aligned}
\right.
\end{equation}
Theorem~\ref{th:euler} shows that \eqref{eq:eulerOG} has a
unique, smooth solution with $a,\nabla \varphi\in Z^\infty$. The main
remark consists in noticing that \eqref{eq:eulerOG} implies that
$(\nabla \varphi,|a|^2)$ has to solve \eqref{eq:euler} ($a_0$ may be
complex-valued): 
Theorem~\ref{th:euler} yields $v,\rho\in Z^\infty$. We can then define
$a$ as the solution to the linear transport equation
\begin{equation*}
  \d_t a+v\cdot \nabla a +\frac{1}{2}a\nabla\cdot v =0\quad
      ;\quad a_{\mid t=0}=a_0.
\end{equation*}
Now $|a|^2$ and $\rho$ solve the same linear transport equation, with
the same initial data, hence $\rho=|a|^2$. Using this information in
the equation for the velocity, define 
\begin{equation*}
  \varphi(t)=\varphi_0 -\int_0^t \(\frac{1}{2}\lvert v(\tau)\rvert^2 +
  f\(|a(\tau)|^2\)\) d\tau. 
\end{equation*}
We easily check that  $\d_t \(\nabla \varphi-v\) =\nabla \d_t \varphi-\d_t v=0$, and
that $(\varphi,a)$ solves \eqref{eq:eulerOG}. 
\smallbreak

Introduce the solution to the linearization of
\eqref{eq:eulerOG}, with an extra source term:
\begin{equation*}
  \left\{
  \begin{aligned}
    &\d_t \varphi^{(1)} +\nabla \varphi \cdot \nabla \varphi^{(1)} +
    2\RE\left(\overline a a^{(1)}\right)f'\left(
    |a|^2\right)=0\quad ;\quad \varphi^{(1)}_{\mid t=0}=0, \\
   &\d_t a^{(1)} +\nabla\varphi\cdot \nabla a^{(1)} + \nabla
   \varphi^{(1)}\cdot \nabla a + \frac{1}{2} a^{(1)}\Delta \varphi
   +\frac{1}{2}a\Delta \varphi^{(1)}      = \frac{i}{2}\Delta a\quad ;
    \quad a^{(1)}_{\mid t=0}=a_1. 
  \end{aligned}
\right.
\end{equation*}
We also check that it has a unique smooth solution, with
$a^{(1)},\nabla \varphi^{(1)}\in Z^\infty$. The main result that we will
invoke is the following:
\begin{proposition}\label{prop:bkw}
  Let $a_0^\eps,a_0,a_1,\varphi_0$ be smooth, with
  $a_0^\eps,a_0,a_1,\nabla \varphi_0\in Z^\infty$. Assume that
  \eqref{eq:CIa} holds. Then 
  \begin{equation}\label{eq:bkw}
    \psi^\eps = \( a e^{i\varphi^{(1)}}+\cO\(\eps\)\)e^{i\varphi/\eps}\text{ in
    }L^\infty([0,\tau];Z^s),\quad \forall \tau <T_{\rm max}\text{ and
    }\forall s\ge 0. 
  \end{equation}
\end{proposition}
This result was established in \cite{Grenier98}
when $Z^s=H^s(\R^d)$, and in \cite{ACIHP} when $Z^s=X^s(\R^d)$. Note
the shift between the order of the approximation between the initial
data (known up to $\cO(\eps^2)$) and the approximation (of order
$\cO(\eps)$ only): this is due to the fact that we consider a regime which
is super-critical as far as WKB analysis is concerned (see
e.g. \cite{CaBook}). In particular, the phase modulation $\varphi^{(1)}$
is a function of $\varphi_0,a_0$ and $a_1$. It is non-trivial in general,
and since we are interested here in real-valued $a_0$, we shall merely
mention two cases (see \cite[pp.~69--70]{CaBook}):
\begin{itemize}
\item If $a_1\not =0$ is real-valued, then $\varphi^{(1)}\not =0$ in general. 
\item If $a_1=0$ (or more generally if $a_1\in i\R$), then $\varphi^{(1)} =0$.
\end{itemize}
To see the first point, it suffices to notice that the equation for
$\varphi^{(1)}$ gives (recall that $f'=1$):
\begin{equation*}
  \d_t \varphi^{(1)}_{\mid t=0} = -2a_0a_1.
\end{equation*}
This shows that for \eqref{eq:madelung}--\eqref{eq:qhd} to yield a
relevant description of the solution to \eqref{eq:nlssemi}, we have to
assume $a_0=\sqrt{\rho_0}$ and $a_1=0$. Otherwise, a phase modulation
is necessary to describe $\psi^\eps$ at leading order, by
\eqref{eq:bkw}, which is 
incompatible with the form \eqref{eq:madelung}, unless the Madelung phase
$\phi^\eps$ admits a corrector of order $\eps$. 
But formal asymptotics in
\eqref{eq:qhd} give $v^\eps = v +\cO(\eps^2)=\nabla\varphi+\cO(\eps^2)$.
Hence the Madelung transform has a chance to give a relevant result only if
$a_1=0$. 
\smallbreak

To check the second point of the above assertion, we set $\alpha = \RE
\(\overline a a^{(1)}\)$. Direct computations show that
$(\varphi^{(1)},\alpha)$ solves, as soon as $a_0\in \R$ and $a_1\in i\R$:
\begin{equation*}
  \left\{
    \begin{aligned}
      &\d_t \varphi^{(1)} + \nabla \varphi\cdot\nabla \varphi^{(1)} + 2\alpha
      =0\quad ;\quad  \varphi^{(1)}_{\mid t=0}=0,\\
& \d_t \alpha +  \nabla \varphi\cdot\nabla \alpha =-\frac{1}{2}\DIV\(
      |a|^2 \nabla \varphi^{(1)}\) -\alpha \Delta \varphi\quad ;\quad
      \alpha_{\mid t=0}=0. 
    \end{aligned}
\right.
\end{equation*}
This is a linear, homogeneous system, with zero initial data, so its
solution is identically zero. 
\smallbreak

To conclude this paragraph, we briefly outline the proof of
Proposition~\ref{prop:bkw}. The approach in \cite{ACIHP} is the same
as in \cite{Grenier98}, with slightly different estimates. For
simplicity, and in view of the above discussion, we assume
$a_0^\eps=a_0$ independent of $\eps$. We write
the solution $\psi^\eps$ as $\psi^\eps =a^\eps e^{i\varphi^\eps/\eps}$
(exact formula), where we impose
\begin{equation}
  \label{eq:emmanuel}
  \left\{
  \begin{aligned}
      &\d_t \varphi^\eps + \frac{1}{2}\lvert \nabla \varphi^\eps \rvert^2 +
      f\(|a^\eps|^2\)=0 \quad ;\quad \varphi^\eps_{\mid t=0}=\varphi_0,\\
&\d_t a^\eps+\nabla \varphi^\eps\cdot \nabla a^\eps
      +\frac{1}{2}a^\eps\Delta \varphi^\eps=i\frac{\eps}{2}\Delta a^\eps \quad 
      ;\quad a^\eps_{\mid t=0}=a_0. 
  \end{aligned}
\right.
\end{equation}
Note that both $\varphi^\eps$ and $a^\eps$ depend on $\eps$, because the
right-hand side of the equation for $a^\eps$ depends on $\eps$, and
because of the coupling between the two equations. Note also that with
this approach, one abandons the possibility of considering a
real-valued amplitude $a^\eps$. 
\smallbreak

It is not hard to construct a solution to \eqref{eq:emmanuel} in
$Z^s$, for $s>d/2+2$, and then check that asymptotic expansions are
available in $Z^s$:
\begin{equation*}
  a^\eps = a+\eps a^{(1)}+\cO\(\eps^2\)\quad ;\quad \varphi^\eps = \varphi+\eps
  \varphi^{(1)} + \cO\(\eps^2\).
\end{equation*}
Back to $\psi^\eps$, this yields Proposition~\ref{prop:bkw}. We see
that the general loss in the precision (from $\cO(\eps^2)$ in the initial data
to $\cO(\eps)$ in the approximation for $t>0$) is due to the division
of $\varphi^\eps$ by $\eps$. Note finally that even though $a_1=0$, one
has $a^{(1)}\not =0$: the corrector $a^{(1)}$ solves a linear
equation, with a purely imaginary (non trivial) source term, and so
$a^{(1)}\in i\R$ is not trivial, while $\varphi^{(1)}=0$ since $\RE
\(\overline a a^{(1)}\)=0$.


\subsubsection{Higher order approximation and formal link with quantum
  hydrodynamics} 
\label{sec:bkw2}

One can actually consider an asymptotic expansion to arbitrary order,
\begin{align*}
  a^\eps &= a+\eps a^{(1)}+\ldots +\eps^N a^{(N)}+\cO\(\eps^{N+1}\),\\
\varphi^\eps &= \varphi+\eps \varphi^{(1)}+\ldots +\eps^N
\varphi^{(N)}+\cO\(\eps^{N+1}\), 
  \quad \forall N\in \N.
\end{align*}
For $j\ge 1$, the coefficient $a^{(j)}$ is given by a linear system
for $(\varphi^{(j)},a^{(j)})$, with source terms involving
$(\varphi^{(k)},a^{(k)})_{0\le k\le j-1}$. In the case $a_0\in \R$ and
$a_1=0$, we know that $\varphi^{(1)}=0$, and $a^{(2)}$ is given by 
\begin{equation*}
  \d_t a^{(2)} + \nabla \varphi\cdot \nabla a^{(2)}
  +\frac{1}{2}a^{(2)}\Delta \varphi + \nabla \varphi^{(2)}\cdot \nabla a
  +\frac{1}{2}a\Delta \varphi^{(2)}=\frac{i}{2}\Delta a^{(1)}\quad ;\quad
  a^{(2)}_{\mid t=0}=0. 
\end{equation*}
We check that in the case $a_0^\eps=a_0\in \R$ (which includes the case 
where Madelung transform is used, $a_0\ge 0$), all the profiles $a$ and 
$a^{(2j)}$, $j\ge 1$,  are real-valued, while $a^{(2j+1)}$, $j\ge 0$, 
are purely imaginary. Moreover, $\varphi^{(2j+1)}=0$ for all $j\in
\N$. This is formally in agreement with 
\eqref{eq:qhd}: indeed, \eqref{eq:qhd} suggests that $\phi^\eps$ and
$\rho^\eps$ have asymptotic expansions of the form
\begin{equation}\label{eq:DAmad}
  \phi^\eps \approx \phi +\eps^2 \phi^{(2)}+\ldots
  +\eps^{2j}\phi^{(2j)}+\ldots\quad ;\quad 
 \rho^\eps \approx \rho +\eps^2 \rho^{(2)}+\ldots
  +\eps^{2j}\rho^{(2j)}+\ldots
\end{equation}
On the other hand, we have
\begin{align*}
  \rho^\eps=\lvert a^\eps\rvert^2 &\approx \left\lvert a+\eps
  a^{(1)}+\ldots\right\rvert^2\\
 &\approx \( a+\ldots + \eps^{2j}a^{(2j)}+\ldots \)^2 -\( \eps
  a^{(1)} +\ldots + \eps^{2j+1}a^{(2j+1)}+\ldots \)^2,
\end{align*}
since the $a^{(2j+1)}$'s are purely imaginary. This is in agreement
with the second formal asymptotics in \eqref{eq:DAmad}. We can check
similarly that  \eqref{eq:DAmad} is in agreement with the higher order
generalization of \eqref{eq:bkw}, in view of the special properties of
the $\varphi^{(j)}$'s and $a^{(j)}$'s pointed out above.  


\subsection{Absence of vacuum before shocks}
\label{sec:q1}

\begin{lemma}\label{lem:novacuum}
  In Theorem~\ref{th:euler}, assume that $\rho_0(x)>0$ for all $x\in
  \R^d$ (absence of vacuum). Then $\rho>0$ on $[0,T_{\rm max})\times
  \R^d$. 
\end{lemma}
\begin{proof}
  As in Section~\ref{sec:euler}, we use the fact that on $[0,T_{\rm
  max})$, the equation for the density is just an ordinary
  differential equation. Introduce the Jacobi determinant
  \begin{equation*}
    J_t(y) ={\rm det}\nabla_y x(t,y),
  \end{equation*}
where $x(t,y)$ is given by \eqref{eq:ray}. We have seen that
$J_t(y)>0$ for $(t,y)\in [0,T_{\rm max})\times
  \R^d$. Change the unknown $\rho$ to $r$, with
  \begin{equation*}
    r(t,y)= \rho\(t,x(t,y)\)J_t(y).
  \end{equation*}
Then for $0\le t<T_{\rm max}$, the continuity equation is equivalent
to: $\d_t r=0$. Therefore,
\begin{equation*}
  \rho(t,x)= \frac{1}{J_t\(y(t,x)\)} \rho_0\(y(t,x)\),
\end{equation*}
where $x\mapsto y(t,x)$ denotes the inverse mapping of $y\mapsto
x(t,y)$. 
\end{proof}
We infer:
\begin{proposition}\label{prop:novacuum}
  Under the assumptions of Theorem~\ref{th:euler}, assume that
  $\rho_0(x)>0$ for all $x\in \R^d$ (absence of vacuum). Let
  $0<T<T_{\rm max}$, and $K$ be a compact set in $\R^d$. There exists
  $\eps(T,K)>0$ such that for $0<\eps 
  \le \eps(T,K)$, $|\psi^\eps|>0$ on $[0,T]\times K$. 
\end{proposition}
\begin{proof}
  Proposition~\ref{prop:bkw} shows that 
  \begin{equation*}
    |\psi^\eps|= |a| +\cO(\eps) \text{ in }
     L^\infty\([0,T]\times\R^d\). 
  \end{equation*}
Note that the constant involved in this $\cO(\eps)$ depends on $T$ in
general. Recalling that $a=\sqrt \rho$, Lemma~\ref{lem:novacuum}
shows that 
$$\min_{(t,x)\in [0,T]\times K}a(t,x)= c(T,K)>0.$$
Now for $0<\eps\le \eps(T,K)\ll 1$, 
\begin{equation*}
  \left\lvert |\psi^\eps|- |a|\right\rvert \le \frac{1}{2}c(T,K) \text{ in }
     L^\infty\([0,T]\times\R^d\),
\end{equation*}
and the result follows.
\end{proof}
In the case $x\in \T^d$, this shows that before the formation of
shocks in the Euler equation, and provided that $\eps$ is sufficiently
small, the amplitude remains positive: the right-hand side of
\eqref{eq:qhd} makes sense. This point was remarked initially in
\cite{PGX93}. Note that the result of \cite{GGZ}
in the one-dimensional case $x\in [0,1]$ shows that suitable boundary
conditions 
lead to the existence of finite time blow-up for
\eqref{eq:qhd}. Therefore, the above result is qualitatively sharp
(qualitatively only, for it might happen that the solution to
\eqref{eq:qhd} remains smooth longer than the solution to
\eqref{eq:euler}). 
\smallbreak

Finally, we show that the compactness assumption in
Proposition~\ref{prop:novacuum} can be removed in the case of the
Gross--Pitaevskii equation. As regards  the nonlinear Schr\"odinger
equation on $\R^d$, this issue seems much more delicate and will not
be addressed in this paper. Assume that the Gross--Pitaevskii equation
is associated with the boundary condition at infinity
\begin{equation*}
  \lvert \psi^\eps(t,x)\rvert \Tend {|x|} \infty 1.
\end{equation*}
Such a condition is used frequently in physics, possibly with a
stronger one, of the form (see e.g. \cite{LinZhang} and references therein)
\begin{equation*}
  \psi^\eps(t,x)- e^{i v^\infty\cdot x/\eps}\Tend {|x|} \infty 0, 
\end{equation*}
for some fixed asymptotic ``velocity'' $v^\infty\in\R^d$. 
\medbreak
Now, as regards the Gross-Pitaevskii equation, 
putting together the continuity of  $\psi^\eps$ over  $[0,T]\times \R^d,$  the compactness
of the time interval $[0,T]$ and Proposition~\ref{prop:novacuum}, we get
\begin{corollary}
  Under the assumptions of Theorem~\ref{th:euler}, assume that
  $\rho_0(x)>0$ for all $x\in \R^d$ (absence of vacuum). Assume
  moreover that the Gross--Pitaevskii equation
is associated with the boundary condition at infinity\footnote{Note that this implies that there exists $c>0$ such that
$\rho_0(x)\ge c$ for all $x\in \R^d$.}
\begin{equation*}
  \lvert \psi^\eps(t,x)\rvert \Tend {|x|} \infty 1.
\end{equation*}
Let $0<T<T_{\rm max}$. There exist
  $\eps(T),c(T)>0$ such that 
  $$|\psi^\eps|\ge c(T)\ \hbox{ on }\ [0,T]\times \R^d, \quad\hbox{for all  }\  0<\eps \le \eps(T).
  $$ 
\end{corollary}


\subsection{Functional spaces associated to the Madelung transform}
\label{sec:q2}

It is rather easy to see that the answer to Question~\ref{q:2} is no,
in general. Consider for $\rho_0$ the function in the Schwartz class
\begin{equation*}
  \rho_0(x)=e^{-\lvert x\rvert^{2k}},\quad k\geq1.
\end{equation*}
At time $t=0$, the quantum pressure (right-hand side of
\eqref{eq:qhd}) grows like $|x|^{2k-1}$ hence  the velocity $v^\eps$ has no
chance to belong to $H^s$ for general initial data in $H^\infty$. Thus,
working in Sobolev spaces for 
general initial data does not make  sense for \eqref{eq:qhd}, while the
results in \cite{Grenier98} show that it is a fairly reasonable
framework to study the semi-classical limit of \eqref{eq:nlssemi}.  
\smallbreak

On the other hand, like for the absence of vacuum, the answer to
Question~\ref{q:3} is positive, at least if we consider some special
boundary conditions at infinity.


\section{Solving the QHD system by a direct
  approach}\label{sec:qhddirect}

In the present section, we describe an efficient method to 
solve directly the hydrodynamic  form of \eqref{eq:nlssemi}
 given by \eqref{eq:qhd}, once performed   
 the  Madelung transform.   
 This method enables us to  study the corresponding 
 initial value problem for  \eqref{eq:qhd} with  data $(v_0,\rho_0)$
 such that $(v_0,\nabla\rho_0)$ has a high order Sobolev regularity and 
  $\rho_0$ is
\emph{positive and bounded away from zero.}
In addition to local-in-time well-posedness results, we get 
 (see Theorems \ref{th:qhd} and \ref{th:qhdbis} below) nontrivial lower bounds
on the first appearance of a zero for the solution, 
 which  are of particular interest for the study of long-wavelength asymptotics 
 if the data are  a  perturbation of a constant state of modulus one.

We here closely follow the approach that has been initiated in \cite{BDS08}. 
To help the reader to compare the present results
with those of the previous section however, we keep on using the semi-classical scaling
given by \eqref{eq:nlssemi} (whereas $\eps=1$ in \cite{BDS08}). 

The use of a suitable  \emph{extended formulation} for \eqref{eq:qhd}
and of weighted Sobolev estimates will be the key 
to our approach. Let us stress that, 
 recently,  similar \emph{extended formulations} 
 have proved to be efficient in other contexts
 for both numerical (see  \cite{CDJLS}) and theoretical purposes. 
 As a matter of fact, in the last paragraph of this section, 
 we shall briefly explain how our approach based on such an extended formulation
 carries over  to the more complicated
 case of Korteweg fluids.

\subsection{Solving the QHD system by means of an extended formulation}

The ``improved'' WKB method that has been described in the previous section
 amounts to writing 
the sought solution $\psi^\eps$ as: $$\psi^\eps=a^\eps e^{i\varphi^\eps/\eps}$$
for some complex valued function $a^\eps$ and real valued function $\varphi^\eps.$  
In this section, we rather start from  the Madelung transform
 $$ \psi^\eps= \sqrt{\rho^\eps}e^{i\phi^\eps}\quad\hbox{where }\ \sqrt{\rho^\eps}=|\psi^\eps|, $$
 then write
 $$
 \psi^\eps=e^{i\Phi^\eps/\eps}\ \hbox{ with }\ 
  \Phi^\eps=\phi^\eps-i\frac\eps2\log\rho^\eps,
 $$
 and consider the redundant system that is satisfied by both $\rho^\eps$ and 
 $z^\eps=\nabla\Phi^\eps=v^\eps+iw^\eps$
 with $v^\eps=\nabla\phi^\eps$ and $w^\eps=-\frac\eps{2\rho^\eps}\nabla\rho^\eps.$
\medbreak
In order to obtain the system for $z^\eps,$ we first
 differentiate the density equation in \eqref{eq:qhd}. This yields
 $$
 \partial_tw^\eps+\nabla(v^\eps\cdot w^\eps)=\frac\eps2\nabla\DIV v^\eps.
 $$
  \smallbreak
 Next,  we notice  that 
 $$
 \frac{\eps^2}2\frac{\Delta\sqrt{\rho^\eps}}{\sqrt{\rho^\eps}}=
 -\frac\eps2\DIV w^\eps+\frac12|w^\eps|^2.
 $$
 In consequence, the equation for $v^\eps$  rewrites 
 $$
 \partial_tv^\eps+v^\eps\cdot\nabla v^\eps-\frac12\nabla|w^\eps|^2
 +\frac\eps2\nabla \DIV w^\eps+\nabla f(\rho^\eps)=0.
 $$
 Of course, $z^\eps$ is a potential vector-field, hence
    $\nabla\DIV z^\eps=\Delta z^\eps$ so that we eventually get 
    the following  ``extended" system for  $(\rho^\eps,z^\eps)$:
 \begin{equation}\label{eq:extended}
\left\{\begin{array}{l}
\displaystyle\partial_tz^\eps+\frac12\nabla(z^\eps\cdot z^\eps)+\nabla f(\rho^\eps)=i\frac{\eps}2\Delta z^\eps,\\[1ex]
 \partial_t\rho^\eps+\DIV(\rho^\eps v^\eps)=0,
 \end{array}\right.
 \end{equation}
 where we agree that  $a\cdot b:=\displaystyle\sum_{j=1}^d a_j b_j$
 for $a$ and $b$ in $\C^d.$

 \smallbreak
 Let us now explain how Sobolev estimates may be 
 derived for $(\rho^\eps,z^\eps)$ in the case where $\rho^\eps=1+b^\eps$ for some $b^\eps$ going
 to $0$ at infinity, and $f(\rho^\eps)=\rho^\eps$
 (to simplify).
 The following computations are borrowed from \cite{BDS08}.
 For notational simplicity, we omit the superscripts $\eps.$
   
In order to get the basic energy estimate,  we compute\footnote{The method 
may seem uselessly complicated. However, the algebraic cancellations
that are going to be used 
remain the same when estimating higher order Sobolev norms.}:
$$
\frac d{dt}\int_{\R^d}\bigl((1+b)|z|^2+b^2\bigr)=
2\biggl(\underbrace{\int_{\R^d}(1+b)\langle z,\partial_t z\rangle}_{I_1}
+\underbrace{\int_{\R^d}  b\,\partial_tb}_{I_2}\biggr)
+\underbrace{\int_{\R^d}\partial_tb\,|z|^2}_{I_3}
$$
where the notation $\langle a,b\rangle=\dis\sum_{j=1}^d \RE a_j\RE b_j
+\IM a_j\IM b_j$ has been used in $I_1.$
\smallbreak
Further computations yield $I_1=I_{1,1}+I_{1,2}+I_{1,3}+I_{1,4}+I_{1,5}$
and $I_2=I_{2,1}+ I_{2,2}$
with
\begin{align*}
&I_{1,1} = -\int \langle z, \nabla b\rangle,
&&I_{2,1} = -\int  b\,\DIV v,\\
&I_{1,2} = -\int b \langle z, \nabla b\rangle,
&&I_{2,2} =- \int  b\, \DIV(bv).\\
&I_{1,3} = \int \langle  z, i\frac\eps2\Delta z\rangle,\\
&I_{1,4} = \int b \langle  z,i\frac\eps2\Delta z\rangle,\\
&I_{1,5} =-\frac12 \int \rho\langle z,
\nabla(z\cdot z)\rangle,
\end{align*}
Using obvious integrations by parts we readily get
$$
I_{1,1}+I_{2,1}=0,\quad I_{1,2}+I_{2,2}=0\ \hbox{ and }\ I_{1,3}=0.
$$
Therefore, integrating by parts in $I_{1,4}$ also, we get
$$
\frac d{dt}\int\bigl((1+b)|z|^2+b^2\bigr)=-\int\eps\langle z,i\nabla z\cdot\nabla b\rangle
-2\int\rho\langle z,\nabla z\cdot z\rangle
+\int \partial_tb\,|z|^2.
$$
For ``general" functions $b$ and $z,$ the appearance 
of the terms $\nabla b$ and $\nabla z$ would preclude any attempt
to ``close" the estimates. In our case however, as the algebraic 
relation $-\eps\nabla b=2\rho w$ holds true, one may avoid this loss of one derivative
for one may write
$$
-\eps\langle z,i\nabla z\cdot\nabla b\rangle-2\rho\langle z,\nabla z\cdot z\rangle
=-2\rho\langle z,\nabla z\cdot v\rangle.
$$
Now, integrating by parts an ultimate time, we conclude that
$$
\frac d{dt}\int\bigl(\rho|z|^2+b^2\bigr)=-\int\rho v\cdot\nabla|z|^2+\int\partial_t b\,|z|^2=0.
$$ 
Hence,  $\displaystyle\int\bigl(\rho|z|^2+b^2\bigr)$ is a conserved quantity. 
\smallbreak

The same algebraic cancellations may be used 
for getting higher order Sobolev (or Besov) estimates. 
Indeed consider an ``abstract" pseudo-differential operator
$A(D)$  (for instance a differential,  
 a fractional derivatives  or a spectral localization operator).
Then one may write 
$$\displaylines{
\frac d{dt}\int\bigl((1+b)|A(D)z|^2+(A(D)b)^2\bigr)\hfill\cr\hfill=
2\biggl(\underbrace{\int(1+b)\langle A(D) z,\partial_t A(D)z\rangle}_{I_1}
+\underbrace{\int A(D) b\,\partial_t A(D)b}_{I_2}\biggr)
+\underbrace{\int\partial_tb\,|A(D)z|^2}_{I_3}.}
$$
We notice that $I_1=I_{1,1}+I_{1,2}+I_{1,3}+I_{1,4}+I_{1,5}$
and $I_2=I_{2,1}+ I_{2,2}$
with
\begin{align*}
&I_{1,1} = -\int \langle A(D) z, \nabla A(D)b\rangle,
&&I_{2,1} = -\int  A(D) b\,\DIV A(D)v,\\
&I_{1,2} = -\int b \langle A(D)z, \nabla A(D)b\rangle,
&&I_{2,2} =- \int  A(D) b\, A(D)\DIV(bv).\\
&I_{1,3} = \int \langle A(D)z, i\frac\eps2\Delta A(D)z\rangle,\\
&I_{1,4} = \int b \langle  A(D)z,i\frac\eps2\Delta A(D) z\rangle,\\
&I_{1,5} =-\frac12 \int \rho\langle A(D)z,A(D)\nabla (z\cdot z)\rangle,
\end{align*}
As above,  obvious integrations by parts ensure that
$I_{1,1}+I_{2,1}=0\ \hbox{ and }\ I_{1,3}=0.$
Next, using again integrations by parts, we notice that
$$
I_{1,2}+I_{2,2}=\int A(D)b\,\DIV[b,A(D)]v.
$$
Finally, integrating by parts in $I_{1,4}$
and using the fact that  $-\eps\nabla b=2\rho w$
yields
$$
I_{1,4}=\int\langle A(D)z,i(\nabla A(D) z)\cdot(\rho w)\rangle,
$$
and we have
$$
I_{1,5}=-\int\!\langle A(D)z, A(D)\nabla z\cdot (\rho v+i\rho w)\rangle 
+\!\int\!\!\rho\langle A(D)z,\nabla A(D)z\cdot z-A(D)(\nabla z\cdot z)\rangle. 
$$
Therefore, using the fact that $\partial_t\rho+\DIV(\rho v)=0,$ we conclude that
$$
2(I_{1,4}+I_{1,5})+I_3=
2\int\rho\langle A(D)z,\nabla A(D)z\cdot z-A(D)(\nabla z\cdot z)\rangle.
$$
Putting all the above equalities together, we thus get
$$\displaylines{
\frac12\frac d{dt}\int\bigl(\rho|A(D)z|^2+(A(D)b)^2\bigr)
=\int A(D)b\,\DIV[b,A(D)]v\hfill\cr\hfill+
\int\rho\langle A(D)z,\nabla A(D)z\cdot z-A(D)(\nabla z\cdot z)\rangle.}
$$
If, say,  $A(D)$ is a fractional  derivatives operator, then 
one may show by means of classical commutator estimates that the right-hand side
may be bounded by 
$$
C\|\rho\|_{L^\infty}\|(Db,Dz)\|_{L^\infty}\|(A(D)b,A(D)z)\|_{L^2}^2.
$$
Therefore,
$$\frac d{dt}\int\bigl(\rho|A(D)z|^2+(A(D)b)^2\bigr)\le C\|\rho\|_{L^\infty}
\|(Db,Dz)\|_{L^\infty}\|(A(D)b,A(D)z)\|_{L^2}^2.
$$ 
Denoting by $E_A^2(t)$ the left-hand side and resorting to Gronwall lemma,
we thus get 
\begin{equation}\label{eq:EA}
E_A(t)\le E_A(0)\exp\biggl(C\int_0^t\|\rho\|_{L^\infty}\|\rho^{-1}\|_{L^\infty}
\|(Db,Dz)\|_{L^\infty}\,d\tau\biggr).
\end{equation}
 It is now clear that whenever 
$Db$ and $Dz$ are bounded in $L^1([0,T];L^\infty)$
and $\rho$ is bounded from below and from above then 
we get a control of $A(D)b$ and $A(D)z$ in $L^\infty([0,T];L^2).$

Taking  $A(D)=\Lambda^s$  and performing a time integration,
\eqref{eq:EA} implies that 
\begin{equation}\label{eq:EAexplo}
\|(b,z)(t)\|_{H^s}\leq
C\biggl(\|(b_0,z_0)\|_{H^s}+\int_0^T\|(Db,Dz)\|_{L^\infty}
\|(b,z)(t)\|_{H^s}\,dt\biggr) 
\end{equation}
for some constant $C=C(s,d,\|\rho^{\pm1}\|_{L^\infty}).$
\smallbreak

So assuming that $s>1+d/2$ and using Sobolev embedding and Gronwall's
inequality,  
one may conclude by elementary methods to the following statement.
\begin{theorem}\label{th:qhd}
Let $s>1+d/2.$ Assume that $\rho_0=1+b_0$ for some $b_0\in H^{s+1}(\R^d)$
such that $1+b_0>0,$ and that $v_0\in H^s(\R^d).$
Then there exists a time 
$$T\ge T_0:=\frac C{\|b_0\|_{H^s}+\eps\|\nabla
  b_0\|_{H^s}+\|v_0\|_{H^{s+1}}}\quad\hbox{with }\  
C=C(s,d,\|\rho_0^{\pm1}\|_{L^\infty})$$ such 
that $\eqref{eq:qhd}$ has a unique solution $(v^\eps,\rho^\eps)$ on $[0,T]\times\R^d$
with $\rho^\eps=1+b^\eps$ bounded away from $0$ and 
$(v^\eps,b^\eps)\in C([0,T];H^{s}\times H^{s+1})\cap C^1([0,T];H^{s-2}\times H^{s-1}).$
\end{theorem}
\begin{remark} Combining basic energy estimates for the wave equation with the above result, 
one may control the discrepancy between $(b,v)$ and the solution to the
acoustic wave equation 
 \begin{equation}\label{eq:acoustic}
\left\{\begin{array}{l}
 \partial_t\dot v+\nabla\dot  b=0\qquad;\qquad \dot v_{|t=0}=v_0,\\[1.5ex]
  \partial_t\dot b+\DIV \dot v=0\qquad;\qquad \dot b_{|t=0}=b_0.
 \end{array}\right.
 \end{equation}
We have, up to time $T_0,$
$$
\|(v^\eps-\dot v,b^\eps-\dot b)(t)\|_{H^{s-2}}\leq C\bigl(t\|(b_0,u_0)\|_{H^{s+1}\times H^s}^2
+\eps^2t\|(b_0,u_0)\|_{H^{s+1}\times H^s}\bigr).
$$

Note also that \eqref{eq:EAexplo} provides a blow-up criterion involving 
the $W^{2,\infty}$ norm of $b$ and the Lipschitz norm of $v.$ 
In particular, this implies that for given data in $H^s$ ($s>1+d/2$), 
the lifespan in $H^s$ is the same as the lifespan in $H^{s'},$ for any $1+d/2<s'<s.$ 
\end{remark}


\subsection{Dispersive properties  
and improved lower bounds for the lifespan}

The system for $(b,v)$ reads
 \begin{equation}\label{eq:qhdbv}
\left\{\begin{array}{l}
 \partial_tv+v\cdot\nabla v+\nabla b=\displaystyle\frac{\eps^2}4\nabla\Bigl(\frac1\rho\Delta b
 -\frac{1}{2\rho^2}|\nabla b|^2\Bigr),\\[1ex] \partial_tb+\DIV v=-\DIV(bv).
 \end{array}\right.
 \end{equation}
Therefore the linearized system about $(0,0)$ is not \eqref{eq:acoustic} but rather
 \begin{equation}\label{eq:qhdlinear}
\left\{\begin{array}{l}
 \partial_t\dot v+\nabla\dot b=\displaystyle\frac{\eps^2}4\nabla\Delta\dot b,\\[1ex]
  \partial_t\dot b+\DIV \dot v=0.
 \end{array}\right.
 \end{equation}
 A straightforward  spectral analysis (based on the Fourier transform) shows that 
 the above linear system behaves as the wave equation
 with speed $1$ for frequencies small with respect to $1/\eps,$
 and as the Schr\"odinger equation with coefficient $\eps/2$
 in  the high frequency regime. In fact, in dimension $d\ge2,$
it is possible to  prove  Strichartz inequalities 
(related to the wave  and  Schr\"odinger equations for low  and high
frequencies, respectively) 
for \eqref{eq:qhdlinear}.
  In the case of small data $(b_0,v_0),$  these inequalities  allow to
  improve the lower bound   for the lifespan (see also \cite{Danchin02}
 where a similar idea has been used in the context of the incompressible limit
 for compressible flows).  
 For the sake of simplicity, let us just state the result in dimension $d\ge4$
 (the reader is referred to \cite{BDS08} for the case $d=2,3$ and for more
 details concerning the approximation of the solution by \eqref{eq:qhdlinear}):
 \begin{theorem}\label{th:qhdbis}
 Under the assumptions of Theorem $\ref{th:qhd}$ with $s>2+d/2,$ then 
 the lifespan $T$ may be bounded from below by 
 $$
 T_1:=\frac C{(\|b_0\|_{H^s}+\eps\|\nabla
  b_0\|_{H^s}+\|v_0\|_{H^{s+1}})^2},
 $$
 and the discrepancy between $(b,v)$ and the solution $(\dot b,\dot v)$
 to $\eqref{eq:qhdlinear}$ with the same data
 may be bounded in terms of $t$ and of the data, up to time $T_1.$
   \end{theorem}


\subsection{Extended formulation for Korteweg  fluids}

Compared to the ``improved'' WKB method, the main drawback of the direct 
approach based on an extended formulation for solving \eqref{eq:qhd} is
that vanishing solutions cannot be handled. 

On the other hand, the direct method is robust enough so as 
to be used to solve locally more complicated models such 
as the following system governing the evolution of  inviscid capillary fluids:
    \begin{equation}\label{eq:capillary}
\left\{\begin{array}{l}
 \partial_tv+v\cdot\nabla v+\nabla f(\rho)=
 \nabla\Bigl(\kappa(\rho)\Delta\rho+\displaystyle\frac12\kappa'(\rho)|\nabla\rho|^2\Bigr),\\[1ex]
  \partial_t\rho+\DIV(\rho v)=0.
  \end{array}\right.
 \end{equation}
Physically, the function $\kappa$ correspond to the capillary coefficient. 
Obviously,  System~\eqref{eq:qhd} is included in
\eqref{eq:capillary}  (take
$\kappa(\rho)=\eps^2/(4\rho)$). 
In the general case, introducing $$a(\rho)=\sqrt{\rho \kappa(\rho)},\quad
w=-\sqrt{\frac{\kappa(\rho)}\rho}\,\nabla\rho\ \hbox{ and }\ z=v+iw,$$
we get the following extended formulation for \eqref{eq:capillary}:
 \begin{equation}\label{eq:capillaryext}
\left\{\begin{array}{l}
 \partial_tz+v\cdot\nabla z+i\nabla z\cdot w+\nabla f(\rho)=i\nabla(a(\rho)\DIV z),\\[1.5ex]
  \partial_t\rho+\DIV(\rho v)=0.
  \end{array}\right.
 \end{equation}
Note that in the potential case (namely ${\rm curl}\, z=0$) then 
$v\cdot\nabla z+i\nabla z\cdot w=\frac12\nabla(z\cdot z).$
Note also that in the general case, the second order term 
$\nabla(a(\rho)\DIV z)$ is degenerate.

The case of System \eqref{eq:qhd}  is particularly simple inasmuch as
 $a$ is the constant function $\eps/2$ and $\nabla\DIV z=\Delta z.$
 
 For general capillarity coefficients, 
 one may  prove a local well-posedness result, 
similar to that  of Theorem \ref{th:qhd}. This has been done in \cite{BDD07}. 
The proof relies on the use of weighted Sobolev estimates, 
with a weight depending both on $\rho$ and on  the order of differentiation.

 As for the QHD system, we expect the potential part of the solution
 to System~\eqref{eq:capillary} 
 to have dispersive properties in dimension $d\ge2.$ The general situation is much more
 complicated however, because those properties are related to those of 
 the  quasilinear Schr\"odinger equation. 
 To  our knowledge,
this aspect has been investigated only very recently in a work by C. Audiard \cite{Audiard2} that concerns 
potential flows.  


\section{Asymptotics for the Gross--Pitaevskii equation}\label{s:JC}

    This section is concerned with the existence and asymptotics of traveling wave solutions for 
    the Gross-Pitaevskii equation
\begin{equation}\label{eq:GP}
i\psi_t+\Delta \psi+(1-|\psi|^2)\psi=0,
\end{equation}
which may be obtained from \eqref{eq:nlssemi} (with
 $f(r)=r-1$), up to the factor $\frac{1}{2}$, after performing the
 change of unknown: 
$$
\psi^\eps(t,x)=\psi(\eps^{-1}t,\eps^{-1}x).
$$
Equation \eqref{eq:GP} is 
associated to the \emph{Ginzburg-Landau energy} (or Hamiltonian): 
\begin{equation}\label{GL}
\cH(\psi) = \frac{1}{2} \int_{\R^d} |\nabla \psi|^2 + \frac{1}{4}
\int_{\R^d} (1 - |\psi|^2)^2.
\end{equation}
As a consequence, 
 in contrast with the cubic NLS equation, the natural energy space 
for \eqref{eq:GP} is not $H^1(\R^d)$ but rather 
$$E(\R^d) = \bigl\{ \psi \in H^1_{\rm loc}(\R^d), \ {\rm s.t.} \ \cH(\psi) < +
\infty \bigr\}\cdotp$$  
As pointed out before, for $\cH(\psi)$ to be finite, $|\psi|$ must, in some sense, tend to $1$ at infinity. 
 This ``nontrivial" boundary condition
provides \eqref{eq:GP} with a richer dynamics  than in the case of null
condition at infinity which, for a defocusing NLS type equation, is
essentially governed by dispersion and scattering. 
For instance, in nonlinear optics, the ``dark solitons" are localized
nonlinear waves (or ``holes")  which exist on a stable continuous wave
background. The boundary condition  $|\psi(t, x)| \rightarrow 1$ at infinity is due to
this nonzero background. In the context of superfluids, $1$ is the 
density of the fluid at infinity. 
\medbreak

Similarly to the energy, the \emph{momentum}
$$\cP(\psi) = \frac{1}{2} \int_{\R^d} \langle i \nabla \psi \ ,
\psi  \rangle,$$
is formally conserved. This quantity is well defined for $\psi \in
H^1(\R^d)$ but not for solutions with a finite Ginzburg-Landau
energy. A major difficulty in the theory of the Gross-Pitaevskii
equation is to find an appropriate definition of the momentum which
leads to a conserved quantity. A natural definition would be  
$$\cP(\psi) = \frac{1}{2} \int_{\R^d} \langle i \nabla \psi \ ,
\psi -1 \rangle,$$
but this would require for instance that $\psi -1 \in L^2(\R^d).$ 

 In any case we will denote by $\p$ the (scalar) first component of
$\cP$ which will play an important role in this section.
\smallbreak
Recall that the use of the Madelung transform\footnote{In this section, we use the normalization 
of \cite{BGS2,BGS1}, instead of $\psi=\sqrt\rho e^{i\phi}.$}
$$\psi = \varrho e^{i\phi},$$
leads to the following hydrodynamic form of the equation for $\varrho$ and $v=2\nabla\phi:$
\begin{equation}
\label{lamadelon}
\left\{ 
\begin{aligned} 
&  \partial_t v + v\cdot\nabla v + 2\nabla \varrho^2 = 2\nabla\biggl(\frac{\Delta\varrho}\varrho\biggr),\\
  &\partial_t \varrho^2 + \div(\varrho^2 v) = 0.
\end{aligned} 
\right.
\end{equation}

As pointed out before, if we discard the right-hand side of the first equation and look at $\varrho^2$ as the density
of a fluid with velocity $v,$ 
then the above system coincides with the Euler equations for a compressible fluid with pressure
law $P(\rho)=\rho^2.$ 
 In particular, the speed of
sound waves near the constant solution $v = 1$ is given by 
$$c_s = \sqrt2.$$
As we will see below, this sound speed (the  value and relevance of which is not so obvious 
if looking at the original Gross-Pitaevskii equation \eqref{eq:GP}),  plays an important role in various aspects of the dynamics of \eqref{eq:GP}.
The value of $c_s$ may be also found by  neglecting the quantum pressure term and linearizing for a
perturbation $\psi = (1 + \tilde{\varrho}) \exp (i \tilde{\phi}).$
This leads to the wave equation: 
$$\partial^2_t \tilde{\varrho} - 2 \Delta \tilde{\varrho} = 0.$$
Note that if  the quantum pressure is included (as in
\eqref{eq:qhdlinear})  then the linearization reads:
$$\partial^2_t \tilde{\varrho} - 2 \Delta \tilde{\varrho} - \Delta^2
\tilde{\varrho} = 0,$$ 
which, roughly,  is the factorization of two linear Schr\"{o}dinger operators.
\medbreak
 The Madelung transform and the hydrodynamic form of \eqref{eq:GP} turn out 
to be of  great interest to study  the Gross-Pitaevskii equation
 with finite Ginzburg--Landau energy  since the solution is expected 
 to have very few ``vortices'' (or cancellations). 
  Even in the ``Euler limit'' that has been presented in Section \ref{s:remi}, 
   one can use it outside the vortices
 (\cite{BOS,BS}) to study the traveling waves of sufficiently
 small velocities. Let us also stress that the hydrodynamic form of the (one
 dimensional) Gross--Pitaevskii equation is needed
 in order to define a generalized momentum in the context of the
 orbital stability of the black solitons (such solitary waves have
 zeroes\dots); see \cite{BGSS1}.  
 
 In the present section, we shall concentrate on the \emph{transonic
   limit} of solutions to the 
 Gross-Pitaevskii equation. 
 We shall first present a result
 pertaining to the asymptotics of traveling waves
 with speed $c$ tending to the sound speed $c_s,$ in the case $d=2,$
 in connexion with the (KP I) equation (see below). 
 Next, for the one-dimensional case, we give an accurate description of  the
 transonic long wave limit of \eqref{eq:GP} in terms of solutions to
 the KdV equation.


 \subsection{The transonic limit of finite energy traveling waves}

   Finite energy traveling wave solutions of \eqref{eq:GP} are solutions
 of the form $\psi(x,t)=\theta(x_1-ct, x^{\perp})$ where
 $\cH(\theta)<+\infty$ and $x^{\perp}$ denotes the transverse
 variables $x_2,\cdots,x_d.$  
 The profile $\theta$ satisfies the following equation: 
\begin{equation}
\label{TWc}
i c \partial_1 \theta + \Delta \theta + \theta (1 - |\theta|^2) = 0.
\end{equation}
A suitable functional setting for the study of such traveling waves
is the space: 
\begin{equation*}
W(\R^d) = \{ 1 \} + V(\R^d),
\end{equation*}
with
 \begin{equation*}
V(\R^d) = \{ \psi: \R^d \mapsto \C,  \ (\nabla \psi, \RE\psi)\! \in\!
L^2(\R^d)^2, \IM\psi\! \in\! L^4(\R^d),  \; \nabla \RE\psi \in
L^{4/3}(\R^d) \}\cdotp
\end{equation*}
Indeed, given that $W(\R^d)$ is a subset of the energy space  $E(\R^d),$
for any data in $W(\R^d),$ Equation 
 \eqref{eq:GP} admits a unique solution. In addition, one may show
 that this solution stays 
 in $W(\R^d)$ (see \cite{PG05,Ger08}). 
 Furthermore,  the quantity $\langle i \partial_1 \psi, \psi-1
\rangle$ is integrable whenever $\psi \in W(\R^d)$, so that the
scalar momentum $\p(\psi)$ is well-defined. This is a consequence of the
identity 
\begin{equation}
\label{identityp}
\langle i \partial_1\psi, \psi-1 \rangle = \partial_1(\RE\psi) \IM\psi - \partial_1(\IM\psi) (\RE\psi - 1),
\end{equation}
and various H\"older's inequalities.

So finally, $\cH$ and $\p$ are continuous on $ W(\R^d)$ and  all finite energy
subsonic solutions to \eqref{TWc}  have to belong to $W(\R^d)$.   
Moreover, if  $\psi\in W(\R^d)$ may be lifted as  $\psi=\varrho \exp (i\phi)$
then   
\begin{equation}
\label{celemoment}
\p(\psi) = \frac{1}{2} \int_{\R^d} \langle i \partial_1 \psi \ , \psi - 1
\rangle = \frac{1}{2} \int_{\R^d} (1 - \varrho^2) \partial_1 \phi. 
\end{equation}
Notice that for maps which may be lifted, with $\varrho\ge
\frac{1}{2}$, the last integral makes sense, even if we assume that
$\psi$ only belongs to the energy space $E(\R^d).$ 
\smallbreak
To simplify the presentation, we shall focus on the simpler case $d=2.$  
We shall denote $x=x_1$ and $y=x^{\perp}=x_2.$

 It is proven in \cite{BGS2} that, for any  $\p>0$,  the following
minimization problem 
\begin{equation}
\label{eminent}
\cH_{\min}(\p) = \inf \{ \cH(\psi), \psi \in W(\R^2), \p(\psi) = \p \},
\end{equation}
has a solution $u_{\p}$ which is a nontrivial traveling wave. We call
it  a \emph{ground state}.

In the rest of this subsection, we focus on the asymptotics $\p$ going
to $0$ for $u_\p,$ in connexion with the 
 Kadomtsev--Petviashvili I (KP I) equation 
 \begin{equation*}
 u_t+uu_x+u_{xxx}-\partial_x^{-1}u_{yy}=0,
 \end{equation*}
where the antiderivative is defined in Fourier variables by
$\widehat{\partial_x^{-1}f}(\xi)=\frac{1}{i\xi}\widehat{f}(\xi).$ 
\medbreak

 The following proposition states that the corresponding 
speed $c(u_\p)$  tends to $c_s$ and gives the first term in the asymptotic
expansion for both $\cH_{\min}(\p)$ and $c(u_\p).$

\begin{proposition}
\label{T2bis}
There exist positive constants $\p_1$, $K_0$ and $K_1$ such that we
have the asymptotic behaviors 
\begin{equation}
\label{estimE2}
\frac{48 \sqrt{2}}{\cS_{KP}^2} \p^3 - K_0 \p^4 \le \sqrt{2} \p -
\cH_{\min}(\p) \le K_1 \p^3,\quad  \forall 0\le \p \le \p_1, 
\end{equation}
where  $\cS_{KP}$ stands for  the  action of the KP I ground state $N$  of
velocity $1,$ that is 
\begin{equation}
\label{action}
\cS_{KP} = \frac{1}{2} \int_{\R^2} (\partial_x N)^2 + \frac{1}{2} \int_{\R^2} (\partial_x^{-1}(\partial_y N))^2 - \frac{1}{6} \int_{\R^2} N^3 + \frac{1}{2} \int_{\R^2} N^2.
\end{equation}

Moreover, the map $u_\p$  has no zeroes and 
there exist some positive constants $\p_2$, $K_2 $ and $K_3$ such that
\begin{equation}
\label{c-estim}
K_2 \p^2 \le \sqrt{2} - c(u_\p) \le K_3 \p^2, \quad\forall 0\le \p < \p_2.
\end{equation}
\end{proposition}
Actually, it was established in \cite{BGS2} that if $\psi$  is a finite energy
traveling wave of sufficiently small energy, then 
$$\frac{1}{2}<|\psi|<1.$$
This implies that small energy traveling waves have no zeroes and thus
can be lifted according to the Madelung transformation. This is in
particular the case of minimizers  $u_\p$ 
corresponding to small enough values of $\p.$

The {\it transonic limit} to KP I for traveling waves is obtained through  the following change of scales:
$$\tilde{x}=\epsilon(\psi)x,\quad
\tilde y=\frac{\epsilon(\psi)^2}{\sqrt{2}}y.$$ 
We then set 
$$\eta=1-|\psi|^2,\quad
N_\psi(x,y)=\frac{6}{\epsilon(\psi)^2}\eta\(\frac{x}{\epsilon(\psi)},
\frac{\sqrt{2}y}{\epsilon(\psi)^2}\).$$ 
It turns out that $N_\psi$ converges to  a traveling wave solution of the
(differentiated) KP I equation \footnote{Actually, to a ground state solution of
  the KP I equation, that is a minimizer of the Hamiltonian with fixed
  $L^2$ norm.} as $\epsilon(\psi)\to 0$, that is it approximately solves
the equation 
\begin{equation}\label{TWKP}
-w_{xx}-w_{yy}+w_{xxxx}+(ww_x)_x=0.
\end{equation}
More precisely, denoting $N_{\p}=1-|u_{\p}|^2$ where
$u_{\p}=|u_{\p}|e^{i\phi_{{\p}}}$ is a minimizer of the energy with
fixed momentum ${\p},$ and
$$\Theta_{{\p}}(x,y)=\frac{6\sqrt{2}}{\epsilon_{{\p}}}\phi_{{\p}}
\(\frac{x}{\epsilon_{{\p}}},\frac{\sqrt{2}y}{\epsilon^2_{{\p}}}\),$$ 
we have (see \cite{BGS1}):
\begin{theorem}
\label{convGPKP}
There exist  a subsequence $(\p_n)_{n \in \N}$ tending to $0$, as $n
\to + \infty$, and  a ground state $w$ of the KP I equation such that 
both $N_{\p_n}$ and $\Theta_{{\p}_n}$ tend to $w$ in $W^{k,q}(\R^2)$
(for any $k\in\N$ and $q\in (1,+\infty]$) as
 $n$ goes to $+\infty.$
 \end{theorem}

\begin{remark} 
Those results on  the transonic limit of solitary waves have been recently extended in \cite{ChMa-p} to the
three-dimensional case  with also  general
nonlinearities   (see also  \cite{Ma-p} for other existence
results). The additional serious difficulty is that  the 
Gross--Pitaevskii ground states solutions are {\it no longer} global
minimizers of 
the energy with fixed momentum and thus not expected to be stable. 
 \end{remark}


\subsection{The unsteady transonic long wave limit of the
   Gross--Pitaevskii equation} 
   
 The Madelung transform is also crucial to derive and justify the
 transonic (weak amplitude, long wave) limit of the
 Gross--Pitaevskii equation. 
 
 With the scaling which is used in this section and for data which are 
perturbations of order $\eps^2$ of a constant state with modulus $1,$ 
Theorem~\ref{th:qhd} and the remark that follows ensure that 
the linear wave equation gives  a good approximation of the solution
for $t=o(\eps^{-3})$.  
In the present paragraph, we describe what happens at next order. 
We shall see in particular that, up to times of order $\eps^{-3},$ the
Korteweg--de Vries (KdV) equation  
 \begin{equation*}
 u_t+uu_x+u_{xxx}=0,
 \end{equation*}
gives an accurate approximation of the one-dimensional Gross-Pitaevskii equation 
 \begin{equation}\label{1DGP}
i\d_t\psi+\partial_{xx} \psi+(1-|\psi|^2)\psi=0
\end{equation}
with data which are small
long-wave perturbations of the constant one, namely $\psi=\varrho
e^{i\phi}$ with  
\begin{equation}\label{eq:data}
 \varrho_0 = \( 1 - \frac{\eps^2}{6} N_\eps^0(\eps x)
 \)^{1/2},\quad 
\phi_0 = \frac{\eps}{6 \sqrt{2}}
\Theta_\eps^0(\eps x), 
\end{equation}
for $0<\eps\ll 1,$  and $N_\eps^0$ and
$W_\eps^0 = \partial_x \Theta_\eps^0$ are uniformly bounded in
some Sobolev space $H^k(\R)$ for sufficiently large $k$.

We refer to \cite{BGSS2,BGSS3} for a detailed analysis and
will only summarize the limit to (long) waves propagating in two
directions and following a coupled system of KdV equations. 

Recall that the one-dimensional Gross--Pitaevskii equation \eqref{1DGP} is
globally well posed in  the Zhidkov type spaces $Y^k,$  
$$Y^k(\R) = \left\lbrace \psi \in L^1_{\text {\rm loc}}(\R; \C), \; \ 1 -
  |\psi| ^2 \in L^2(\R),\quad\partial_x \psi \in H^{k - 1}(\R)
\right\rbrace,$$ 
for any integer  $k \ge 1$. 

Moreover the Ginzburg-Landau energy $\cH(\psi(t))$ is conserved by the flow
and, provided $\cH(\psi_0)<\frac{2\sqrt{2}}{3},$ the corresponding solution
$\psi(t)$ does not vanish so that one may write $\psi= \varrho \exp (i \phi)$, for some continuous function $\phi.$ 
\medbreak
We consider data as in \eqref{eq:data} with small enough $\eps$ and assume in addition that
\begin{equation*}
\| N_\eps^0 \|_{\mathcal M(\R)}+ \| \partial_x \Theta_\eps^0
\|_{\mathcal M(\R)} < + \infty. 
\end{equation*}
Here, $\|\cdot \|_{\mathcal M(\R)}$ denotes the norm defined on
$L_{\rm loc}^1(\R)$ by 
\begin{equation*}
\| f \|_{\mathcal M(\R)} = \underset{(a, b) \in \R^2}{\sup} \bigg|
\int_a^b f(x) dx \bigg|. 
\end{equation*}
We next introduce the \emph{slow coordinates}
$$x^- = \varepsilon (x + \sqrt{2} t), \ x^+ = \varepsilon (x -
\sqrt{2} t), \ {\text{ and}} \ \tau = \frac{\varepsilon^3}{2 \sqrt{2}}
t.$$ 
The definition of the coordinates $x^-$ and $x^+$ corresponds to
reference frames traveling to the left and to the right, respectively,
with speed $\sqrt{2}$ in the original coordinates $(t,x)$. We define
accordingly the rescaled functions $N_\eps^\pm$ and $\Theta_\eps^\pm$
as follows: 
\begin{equation}
\label{slow-var}
\left\{
\begin{aligned}
N_\eps^\pm(\tau,x^\pm) & = \frac{6}{\eps^2} \eta(t,x) =
\frac{6}{\eps^2} \eta \( \frac{2 \sqrt{2} \tau}{\eps^3},\frac{x^\pm}{\eps} \pm \frac{4 \tau}{\eps^3}
 \)\quad\hbox{ with }\ \eta = 1 - \varrho^2,\\ 
\Theta_\eps^\pm(\tau,x^\pm) & = \frac{6 \sqrt{2}}{\eps} \phi(t,x) =
\frac{6 \sqrt{2}}{\eps} \phi \(\frac{2 \sqrt{2} \tau}{\eps^3} , \frac{x^\pm}{\eps} \pm \frac{4
\tau}{\eps^3}\).
\end{aligned}
\right.
\end{equation}
Setting
\begin{equation}
\label{eq:uv}
\left\{
\begin{aligned}
U_\eps^-(\tau,x^-) = \frac{1}{2} \( N_\eps^-(\tau,x^-)
+ \partial_{x^-} \Theta^-(\tau,x^-) \),\\ 
U_\eps^+(\tau,x^+) = \frac{1}{2} \( N_\eps^+(\tau,x^+)
- \partial_{x^+} \Theta^+(\tau,x^+) \), 
\end{aligned}
\right.
\end{equation}
the main result is (see \cite{BGSS3} for details):
\begin{theorem}
\label{cochondore}
Let $k \ge 0$ and $\eps > 0$ be given. Assume that the initial data
$\psi_0$ belongs to $Y^{k + 6}(\R)$ and satisfies the assumption 
\begin{equation*}
\| N_\eps^0 \|_{\mathcal M(\R)} + \|\partial_x \Theta_\eps^0
\|_{\mathcal M (\R)} + \| N_\eps^0 \|_{H^{k + 5}(\R)} + \eps
\| \partial_x^{k + 6} N_\eps^0 \|_{L^2(\R)} + \|\partial_x
\Theta_\eps^0 \|_{H^{k + 5}(\R)} \le K_0. 
\end{equation*}
Let $\mathcal U^-$ and $\mathcal U^+$ denote the solutions to the
Korteweg--de Vries equations 
\begin{equation*}
\partial_\tau \mathcal U^- + \partial_{x^-}^3 \mathcal U^- + \mathcal
U^- \partial_{x^-} \mathcal U^- = 0, 
\end{equation*}
and
\begin{equation*}
\partial_\tau \mathcal U^ + - \partial_{x^+}^3 \mathcal U^+ - \mathcal
U^+ \partial_{x^+} \mathcal U^+ = 0, 
\end{equation*}
with the same initial value as $U_\eps^-$ and $U_\eps^+$,
respectively. Then, there exist positive constants $\eps_1$ and $K_1$,
depending only on $k$ and $K_0$, such that 
\begin{equation*}
\| U_\eps^-(\tau,\cdot) - \mathcal U^-(\cdot, \tau) \|_{H^k(\R)} +
\|U_\eps^+(\tau,\cdot) - \mathcal U^+(\cdot, \tau) \|_{H^k(\R)} \le
K_1 \eps^2 \exp \( K_1 |\tau|\),
\end{equation*}
for any $\tau \in \R$ provided $\eps \le \eps_1$.
\end{theorem}
We now turn to the two (or higher) dimensional case,
 which is  studied in \cite{ChRo10} for a general  nonlinear Schr\"{o}dinger equation of the form
similar to \eqref{eq:GP}
\begin{equation*}
i\d_t\psi+\Delta \psi= f(|\psi|^2)\psi,
\end{equation*}
where $f$ is smooth and satisfies $f(1)=0,$ $f'(1)>0.$

One also uses a ``weakly transverse transonic" scaling, namely
$$T=c\eps^3t, \quad X_1=\eps (x_1-ct),\quad X_j=\eps
^2x_j, \quad j=2,\cdots,d.$$ 
After performing the ansatz 
$$\psi^{\eps}(t,X)=\( 1+\eps ^2 A^{\eps}(t,X)\) \exp (i\eps
\phi^\eps(t,X)),$$ 
the hydrodynamic reformulation of the Gross--Pitaevskii equation is
used to recast the problem as a singular limit for an hyperbolic
system in the spirit of \cite{Grenier98}. 
Then smooth $H^s$ solutions are proven to exist on an interval
independent of the small parameter $\eps$. Passing to the limit by a
compactness argument yields the convergence of the solutions to that
of the KP-I equation. Note however that this method does  not provide
a convergence rate with respect to $\eps$, contrary to the KdV case
considered above.  

In comparison, for such data, Theorem \ref{th:qhdbis} would ensure
that the linear system 
\eqref{eq:qhdlinear} gives a good description of the solution 
only for times that are $o(\eps^{-3})$
 (see the introduction of \cite{BDS08} for more details). 


\section{Global existence of weak solutions to a quantum fluids
  system} \label{sec:global}
We aim at  providing  an elementary  proof of the result
by P.~Antonelli and P.~Marcati 
in \cite{AnMa,AM09} concerning global finite energy weak solutions
to the QHD system. Here is the  statement: 

\begin{theorem}\label{th:AM}
Let the initial data $(\rho_0,\Lambda_0)\in W^{1,1}\times L^2$ be 
``well-prepared'' in the sense that 
there exists some wave function $\psi_0\in H^1$ such that 
$$
\rho_0=|\psi_0|^2\quad\hbox{and}\quad
J_0:=\sqrt{\rho_0}\Lambda_0=\IM(\bar\psi_0\nabla\psi_0).
$$
Assume that $f(r)=r^\sigma$ for some integer $\sigma$
such that $W^{1,1}(\R^d)\hookrightarrow L^{\sigma+1}(\R^d).$

There exist some vector-field $\Lambda\in L^\infty(\R;L^2)$
and some nonnegative function $\rho\in L^\infty(\R;L^1\cap L^{\sigma+1})$
with  $\nabla\sqrt\rho\in L^\infty(\R;L^2)$ such that
the following system holds true in the distributional 
sense\footnote{In the smooth non-vanishing case, the right-hand side of the second equation  
coincides with that of the velocity equation multiplied by $\rho$ in \eqref{eq:qhd} with $\eps=1$, and 
the third equation just means that there exists some function $\phi$ such that $J=\rho\nabla\phi.$}:
\begin{equation}
  \label{eq:qhdgen}
  \left\{
    \begin{aligned}
   &\partial_t\rho+\div J=0,\\[1ex]
&\partial_tJ+\div(\Lambda\otimes\Lambda)+\nabla(P(\rho))=\frac14\Delta\nabla\rho
-\div(\nabla\sqrt\rho\otimes\nabla\sqrt\rho),\\[1ex]
&\partial_j J^k-\partial_kJ^j=2\Lambda^k\partial_j\sqrt\rho
-2\Lambda^j\partial_k\sqrt\rho\quad\hbox{for all }
(j,k)\in \{1,\cdots,d\}^2,\\[1ex]
&(\rho,J)_{|t=0}=(\rho_0,J_0),   
    \end{aligned}
\right.
\end{equation}
with $J:=\sqrt\rho\Lambda$ and $P(\rho):=\rho f(\rho)-F(\rho)$
with $F(\rho)=\int_0^\rho f(\rho')\,d\rho'.$
\medbreak
In addition, the energy 
$$
\int_{\R^d} \Bigl(\frac12|\Lambda|^2+\frac12|\nabla\sqrt\rho|^2+F(\rho)\Bigr)
$$
is conserved for all time. 
\end{theorem}
\begin{proof}
Let us first prove the statement in the smooth case, namely 
we assume that the data $\psi_0$ is in $H^s$ for some large enough $s.$
It is well known that 
\begin{equation}\label{eq:NLS}
  i\d_t \psi+ \frac{1}{2}\Delta \psi = f\(|\psi|^2\)\psi\quad ;\quad \psi_{\mid
    t=0}=\psi_0
\end{equation}
 has a unique solution $\psi$ in 
$C(\R;H^s)$  whenever $s\ge 1$ (see e.g. \cite{CazCourant,LP,Tao}) and that
\begin{equation}\label{eq:energypsi}
\forall
t\in\R,\;\int_{\R^d}\(\frac12|\nabla\psi(t)|^2+F\(|\psi(t)|^2\)\)= 
\int_{\R^d}\(\frac12|\nabla\psi_{0}|^2+F\(|\psi_{0}|^2\)\).
\end{equation}
Let us set $\rho:=|\psi|^2$ and 
\begin{equation}\label{eq:phi}
\phi(x):=\left\{\begin{array}{lll} |\psi(x)|^{-1}\psi(x)&\hbox{if}&\psi(x)\not=0,\\[1ex]
0&\hbox{if}& \psi(x)=0.\end{array}
\right.\end{equation}
We claim that 
\begin{equation}\label{eq:sqrtrho}
 \nabla\sqrt\rho=\RE(\bar\phi\nabla\psi)\quad\hbox{a. e. }
\end{equation}
Indeed, for $\eps>0,$  let us 
set $\phi_\eps:=\psi/\sqrt{|\psi|^2+\eps^2}.$
Then  $(\phi_\eps)$ converges pointwise to $\phi$
and an easy computation shows that 
\begin{equation}\label{eq:cu}
\bar\phi_\eps\psi\Tend\eps 0 \sqrt\rho\quad\hbox{uniformly}.
\end{equation}
Next, we compute
\begin{equation}\label{eq:dec}
\nabla(\bar\phi_\eps\psi)=\RE(\bar\phi_\eps\nabla\psi)+
\RE(\psi\nabla\bar\phi_\eps).  
\end{equation}
The first term in the right-hand side converges pointwise
to $\RE(\bar\phi\nabla\psi)$ hence in $L^1_{\rm loc}$
owing to Lebesgue's theorem as it is bounded by $|\nabla\psi|.$ 

As for the last term, we have
$$
\RE(\psi\nabla\bar\phi_\eps)=\frac{\eps^2}{(\eps^2+|\psi|^2)^{3/2}}
\RE(\bar\psi\nabla\psi). 
$$
It is clear that the right-hand side converges pointwise to $0$
and is bounded by $|\nabla\psi|.$ Hence it 
also converges to $0$ in $L^1_{\rm loc}.$

So finally, putting these two results together with \eqref{eq:cu} and \eqref{eq:dec}, 
one may conclude to \eqref{eq:sqrtrho}.
\smallbreak

 Let $\Lambda:=\IM(\bar\phi\nabla\psi)$ and $J:=\sqrt\rho\Lambda.$
 We claim that $(\rho,\Lambda,J)$ satisfies \eqref{eq:qhdgen}.
 Indeed, from \eqref{eq:NLS}, we see that 
 $$
 \partial_t(|\psi|^2)=2\RE(\bar\psi\partial_t\psi)
 =-\IM(\bar\psi\Delta\psi)
 =-\IM\div(\bar\psi\nabla\psi),
 $$
 hence
 \begin{equation}\label{eq:rho}
 \partial_t\rho+\div J=0.
 \end{equation}
Next, we compute
 $$\partial_j J^k-\partial_kJ^j=\IM\bigl(\partial_j(\bar\psi\partial_k\psi)
 -\partial_k(\bar\psi\partial_j\psi)\bigr)
 =2 \IM\bigl(\partial_j\bar\psi\partial_k\psi\bigr).
 $$
 Recall (see e.g. Theorem 6.19 in \cite{LL})
 that 
\begin{equation}\label{eq:pseudosard}
 \nabla\psi=0 \quad\hbox{a. e.  on}\quad \psi^{-1}(\{0\})
 \end{equation}
 whenever $\psi$ is locally in $W^{1,1}.$
 \medbreak

Therefore, given that $|\phi|=1$ on $\psi^{-1}(\C\setminus\{0\}),$
  one may write a. e. 
 $$
 \IM\bigl(\partial_j\bar\psi\partial_k\psi\bigr)
 = \IM\bigl(\phi\partial_j\bar\psi\,\bar\phi\partial_k\psi\bigr)
=\RE(\phi\partial_j\bar\psi)\,\IM(\bar\phi\partial_k\psi)
+\RE(\bar\phi\partial_k\psi)\,\IM(\phi\partial_j\bar\psi).
$$
So   we eventually get
 \begin{equation}\label{eq:curl}
 \partial_j J^k-\partial_kJ^j=2\partial_j\sqrt\rho\:\Lambda^k
 -2\partial_k\sqrt\rho\:\Lambda^j.
\end{equation}
Next, we compute 
$$
\partial_tJ=\IM\bigl(\partial_t\bar\psi\,\nabla\psi+\bar\psi\nabla\partial_t\psi\bigr).
$$
Hence, using the equation satisfied by $\psi,$ we get
$$
\partial_tJ=
\frac12\underbrace{\RE\bigl(\bar\psi\nabla\Delta\psi-\Delta\bar\psi\nabla\psi\bigr)}_{A}
+\underbrace{\RE\bigl(f(|\psi|^2)\bar\psi\nabla\psi-\bar\psi\nabla(f(|\psi|^2)\psi)\bigr)}_{B}.
$$
If $\psi$ is $C^1$ and the function $f$ has a derivative at every point of $\R^+,$ then 
straightforward computations show that 
\begin{equation}\label{eq:B}
B=-\rho f'(\rho)\nabla\rho=-\nabla(P(\rho)).
\end{equation}
Next, we see that (still in the smooth case)
$$
\nabla\Delta|\psi|^2=4\RE\bigl(\nabla^2\psi:\nabla\bar\psi\bigr)
+2\RE\bigl(\nabla\psi\Delta\bar\psi\bigr)
+2\RE\bigl(\bar\psi\nabla\Delta\psi\bigr).$$
Hence 
$$
A=\frac12\nabla\Delta|\psi|^2
-2\RE\bigl(\Delta\bar\psi\nabla\psi\bigr)-2\RE\bigl(\nabla^2\psi:\nabla\bar\psi\bigr).
$$
So we get
\begin{equation}\label{eq:A}
\frac12 A=\frac14\nabla\Delta|\psi|^2-\div\RE(\nabla\bar\psi\otimes\nabla\psi).
\end{equation}
Now, using again \eqref{eq:pseudosard}, one may write at
almost every point of $\R^d,$ 
$$\begin{array}{lll}
\RE(\nabla\bar\psi\otimes\nabla\psi)&=&
\RE(\phi\nabla\bar\psi\otimes\bar\phi\nabla\psi),\\[1ex]
&=&\RE(\bar\phi\nabla\psi)\otimes\RE(\bar\phi\nabla\psi)
+\IM(\bar\phi\nabla\psi)\otimes\IM(\bar\phi\nabla\psi).\end{array}
$$
Therefore, we have
\begin{equation}\label{eq:Psi}
\RE(\nabla\bar\psi\otimes\nabla\psi)=\nabla\sqrt\rho\otimes\nabla\sqrt\rho
+\Lambda\otimes\Lambda.
\end{equation}
Putting this together with \eqref{eq:B} and \eqref{eq:A},
 one may conclude that
\begin{equation}\label{eq:J}
\partial_tJ+\div(\Lambda\otimes\Lambda)+\nabla(P(\rho))
=\frac14\Delta\nabla\rho-\div(\nabla\sqrt\rho\otimes\nabla\sqrt\rho).
\end{equation}
Of course, as owing to \eqref{eq:pseudosard}
\begin{equation}\label{eq:mod}
|\nabla\psi|^2
=(\RE(\bar\phi\nabla\psi))^2+(\IM(\bar\phi\nabla\psi))^2
=|\nabla\sqrt\rho|^2+|\Lambda|^2\quad\hbox{a. e.}, 
\end{equation}
the energy equality for $\psi$ recasts in 
\begin{equation}\label{eq:energyJ}
\int_{\R^d}
\Bigl(\frac12|\Lambda(t)|^2+\frac12|\nabla\sqrt{\rho(t)}|^2+F\(\rho(t)\)\Bigr) 
=\int_{\R^d}
\Bigl(\frac12|\Lambda_{0}|^2+\frac12|\nabla\sqrt{\rho_{0}}|^2+F(\rho_{0})\Bigr). 
\end{equation}
This completes the proof in the smooth case.
\medbreak

Let us now  treat the rough case where $\psi_0$ belongs only to $H^1.$
Then we  fix some sequence $(\psi_{0,n})_{n\in\N}$ of functions in $H^s$ (with $s$ large)
converging to $\psi_0$ in $H^1.$
Let us denote by $\psi_n$ the solution of \eqref{eq:NLS}
in $C(\R;H^s)$ corresponding to the data  $\psi_{0,n},$
and by $\psi\in C(\R;H^1)$ the solution to \eqref{eq:NLS}
with data $\psi_0.$

From the first part of the proof, we know that there 
exists some sequence $(\phi_n)_{n\in\N}$ of functions
with modulus at most $1$ such that if we set 
 $\Lambda_n:=\IM(\bar\phi_n\nabla\psi_n),$
 $\rho_n:=|\psi_n|^2$ and $J_n:=\sqrt{\rho_n}\Lambda_n$
then 
$$
\nabla\sqrt{\rho_n}=\RE(\bar\phi_n\nabla\psi_n) \quad\hbox{in }\ L^2,
$$
and 
 $(\rho_{n},\Lambda_{n},J_{n})$ satisfies~\eqref{eq:qhdgen}
with data $(\rho_{0,n},J_{0,n}),$
together with the energy equality
\begin{eqnarray}\label{eq:energyJn}
&&\int_{\R^d}
\Bigl(\frac12|\Lambda_n(t)|^2+\frac12|\nabla\sqrt{\rho_n(t)}|^2+
F\(\rho_n(t)\)\Bigr) 
\\&&\hspace{3cm}=\int_{\R^d}
\Bigl(\frac12|\Lambda_{0,n}|^2+\frac12|\nabla\sqrt{\rho_{0,n}}|^2
+F(\rho_{0,n})\Bigr). 
\nonumber\end{eqnarray}
We now have to prove the convergence of $(\rho_n,\Lambda_n,J_n)$
to some solution $(\rho,\Lambda,J)$ of \eqref{eq:qhdgen} satisfying the energy
equality. 

On the one hand, standard  stability estimates (based
on Strichartz inequalities) guarantee that
\begin{equation}\label{eq:convpsi}
\psi_n\longrightarrow\psi\quad\hbox{in}\quad L^\infty_{\rm loc}(\R;H^1).
\end{equation} 
On the other hand, because
$(\phi_n)_{n\in\N}$ is bounded by $1,$ 
it converges  (up to extraction) in $L^\infty$ weak * 
to some function $\phi$ such that $\|\phi\|_{L^\infty}\le1.$
As $\nabla\psi_n$ converges strongly to $\nabla\psi$ in $L^2,$
this implies that 
\begin{equation}\label{eq:weakconv}
\bar\phi_n\nabla\psi_n\rightharpoonup\bar\phi\nabla\psi\quad\hbox{in }\ L^2.
\end{equation}
In turn, as obviously  $\sqrt{\rho_n}\rightarrow\sqrt\rho$ in 
$L^2$ and as  $\nabla\sqrt{\rho_n}=\RE(\bar\phi_n\nabla\psi_n),$ we deduce   that
$$
\nabla\sqrt\rho=\RE(\bar\phi\nabla\psi)\quad\hbox{and}\quad
\nabla\sqrt{\rho_n}\rightharpoonup \nabla\sqrt\rho\ \hbox{ in }\ L^2.
$$
Given that $\Lambda_n=\IM(\bar\phi_n\nabla\psi_n),$ 
\eqref{eq:weakconv} also ensures  that 
$$
\Lambda_n\rightharpoonup \Lambda:=\IM(\bar\phi\nabla\psi)\quad\hbox{in }\ L^2.
$$
In order to establish that strong convergence in $L^2$ holds true, 
it suffices to show that (up to an omitted extraction) 
\begin{equation}\label{eq:strong}
\|\nabla\sqrt{\rho_n}\|_{L^2}\rightarrow\|\nabla\sqrt\rho\|_{L^2}
\quad\hbox{and}\quad
\|\Lambda_n\|_{L^2}\rightarrow\|\Lambda\|_{L^2}.
\end{equation}
On the one hand, the weak convergence ensures that
$$
\|\nabla\sqrt\rho\|_{L^2}\le \liminf\|\nabla\sqrt{\rho_n}\|_{L^2}
\quad\hbox{and}\quad
\|\Lambda\|_{L^2}\le\liminf \|\Lambda_n\|_{L^2};
$$
on the other hand, given that $(\nabla\psi_n)_{n\in\N}$ converges strongly to $\nabla\psi$
in $L^2$ and that \eqref{eq:mod} holds true for $\psi$ and $\psi_n,$ we may write
\begin{align*}
\|\nabla\sqrt\rho\|_{L^2}^2+\|\Lambda\|_{L^2}^2&=\|\nabla\psi\|_{L^2}^2,\\[1ex]
&=\lim_{n\rightarrow+\infty}\|\nabla\psi_n\|_{L^2}^2,\\[1ex]
&=\lim_{n\rightarrow+\infty}\bigl(\|\nabla\sqrt{\rho_n}\|_{L^2}^2+\|\Lambda_n\|_{L^2}^2\bigr),\\[1ex]
&\ge\bigl(\liminf \|\nabla\sqrt{\rho_n}\|_{L^2}\bigr)^2+\bigl(\liminf \|\Lambda_n\|_{L^2}\bigr)^2.
\end{align*}
Therefore \eqref{eq:strong} is satisfied.
As a conclusion, we thus  have established that 
$$
\sqrt{\rho_n}\rightarrow\sqrt\rho\ \hbox{ in }\ H^1\quad\hbox{and}\quad
\Lambda_n\rightarrow\Lambda\ \hbox{ in } L^2.
$$
Of course, this implies that $J_n\rightarrow J:=\sqrt\rho\Lambda$ in $L^1$
so it is easy to pass to the limit in \eqref{eq:qhdgen} and in the energy equality \eqref{eq:energyJ}. 
The details are left to the reader.
\end{proof}

\appendix
\section{Conservation laws}

In this Appendix we review conservations laws for the nonlinear
Schr\"odinger, QHD and compressible Euler equations. 
Even though most of the results are classical (as concerns
the Schr\"odinger equation, they may be found in 
the textbooks \cite{Sulem,Tao} for instance; links between
Schr\"odinger and Euler conservation laws may be found in \cite{CaBook}), 
we believe the relationships between the aforementioned
equations to be of interest. In addition,  
those conservation laws are still meaningful for  the less classical
framework of general Korteweg fluids.

\subsection{The case of  Schr\"odinger,  QHD  and  compressible Euler equations}

For the time being, we consider the following system 
\begin{equation}\label{QHD_eps}
  \left\{
    \begin{aligned}
      &\d_t v +v\cdot \nabla v+
      \nabla f\(\rho\)= \frac{\eps^2}{2}\nabla\( \frac{\Delta\(\sqrt
      \rho\)}{\sqrt \rho}\)\quad ;\quad 
      v_{\mid t=0} =v_0,\\
& \d_t \rho  + \DIV \(\rho v\) =0\quad ;\quad \rho_{\mid
      t=0}=\rho_0,
    \end{aligned}
\right.
\end{equation}
which is the QHD system if $\eps>0,$ and the compressible
Euler equation if $\eps=0$,
and the nonlinear  Schr\"odinger equation:
\begin{equation}
  \label{NLS_eps}
  i\eps\d_t\psi+\frac{\eps^2}2\Delta\psi=f(|\psi|^2)\psi\quad;\quad
\psi_{\mid t=0}=\psi_0.
\end{equation}
Recall that for $\eps>0,$ one may pass formally from  \eqref{NLS_eps} to
\eqref{QHD_eps}  by setting 
$$
\psi=\sqrt\rho\,e^{i\phi/\eps}\quad\hbox{and}\quad
v=\nabla \phi.
$$
In what follows, the function $f$ is assumed to be 
continuous on $\R^+$ and, say,
$C^1$ on $(0,+\infty[,$  standard cases being 
$f(r)=r^\sigma$  and $f(r)=r-1.$
We denote by $F$ the anti-derivative of $f$ which vanishes at  $0,$
and set $P(\rho)=\rho f(\rho)-F(\rho).$
As pointed out before, from a physical viewpoint, $P$ is the pressure.

\smallbreak
The first part of the appendix aims at listing (and deriving formally) 
the classical conservation laws for  \eqref{NLS_eps} and  \eqref{QHD_eps}.


\subsubsection*{Phase invariance}

For every $\alpha\in\R,$ one has 
$$\psi\ \hbox{ solution of  }\  \eqref{NLS_eps} \iff 
e^{i\alpha} \psi \hbox{ solution of }\  \eqref{NLS_eps}.
$$
The phase invariance is not seen at the level of \eqref{QHD_eps}
(this amount to changing  $\phi$ into  $\phi+\eps\alpha$).  

By Noether's theorem or by an easy computation, this leads to
the conservation of \emph{mass}:
\begin{equation}\label{eq:M}
\cM:=\int|\psi|^2\,dx=\int\rho\,dx.
\end{equation}


\subsubsection*{Time translation invariance}

For every $\tau\in\R,$ one has  
$$\psi(t,x)\ \hbox{ solution of }\  \eqref{NLS_eps} \iff 
\psi(t+\tau,x) \hbox{ solution of }\  \eqref{NLS_eps}.
$$

By the same time translation, this is expressed at the level of
  \eqref{QHD_eps}
and leads to the conservation of the \emph{energy} (or \emph{Hamiltonian}):
\begin{equation}\label{eq:E}
\cH:=\int\Bigl(\frac{\eps^2}2|\nabla\psi|^2+F(|\psi|^2)\Bigr)\,dx
=\int\Bigl(\frac12\rho|v|^2+\frac{\eps^2}2|\nabla\sqrt\rho|^2+F(\rho)\Bigr)\,dx.
\end{equation}


\subsubsection*{Space translation invariance}

For every  $x_0\in\R^d,$ one has 
$$\psi(t,x)\ \hbox{ solution to }\  \eqref{NLS_eps} \iff 
\psi(t,x+x_0) \hbox{ solution to }\  \eqref{NLS_eps}.
$$
 By the same space  translation, this is expressed at the level of
 \eqref{QHD_eps} 
and leads to the  conservation of \emph{momentum}:
\begin{equation}\label{eq:P}
\cP:=\IM\int\eps\bar\psi\nabla\psi\,dx=\int\rho v\,dx.
\end{equation}


\subsubsection*{Invariance by spatial  rotation}

Let  $R$ be a spatial  rotation. Then
$$\psi(t,x)\ \hbox{ solution of }\  \eqref{NLS_eps} \iff 
\psi(t,Rx) \hbox{ solution of }\  \eqref{NLS_eps}.
$$
By the same spatial rotation on a  solution 
of \eqref{QHD_eps}
this leads to the conservation of  \emph{angular momentum}, which we
write in the case of $\R^3$ for the sake of simplicity:
\begin{equation}\label{eq:cA}
\cA:=\IM\int x\wedge \eps\bar\psi\nabla\psi\,dx=\int x\wedge\rho v\,dx.
\end{equation}


\subsubsection*{Galilean invariance}

For every  $\xi_0\in\R^d,$ one has 
$$\psi(t,x)\ \hbox{ solution of }\  \eqref{NLS_eps} \!\iff\! 
e^{-i\xi_0\cdot x} e^{-i\frac{t\eps}2|\xi_0|^2}
\psi(t,x+\eps\xi_0t) \hbox{ solution of }\  \eqref{NLS_eps}.
$$
For \eqref{QHD_eps}, this implies
$$
(v,\rho)(t,x)\ \hbox{ solution} \iff 
\bigl(v(t,x+\eps\xi_0 t)-\eps\xi_0,\rho(t,x+\eps\xi_0 t)\bigr)
\ \hbox{ solution},
$$
and leads to 
\begin{equation}\label{eq:X}
\frac{d\cX}{dt}=\cP\quad\hbox{with}\ \cX:=\int x|\psi|^2\,dx=\int x\rho\,dx.
\end{equation}


\subsubsection*{Scale invariance } 

When  $f(r)=r^\sigma,$ one has for all  $\lambda>0,$
$$\psi(t,x)\ \hbox{ solution of }\  \eqref{NLS_eps} \iff 
\lambda^{1/\sigma}\psi(\lambda^2t,\lambda x) \hbox{ solution of }\
\eqref{NLS_eps}, 
$$
which implies for \eqref{QHD_eps}  
$$
(v,\rho)(t,x)\ \hbox{ solution}\iff 
\bigl(\lambda v,\lambda^{2/\sigma}\rho\bigr)(\lambda^2t,\lambda x)
\ \hbox{ solution.}
$$
The associated conservation law is (recall that $d$ stands for
the space dimension):
$$
\frac{d\cF}{dt}=\int\Bigl(\eps^2|\nabla\Psi|^2+dP(|\psi|^2)\Bigr)\,dx
=2\cH+\int(dP-2F)(|\psi|^2)\,dx
$$
or, 
$$
\frac{d\cF}{dt}=\int\Bigl(\rho |v|^2+dP(\rho)\Bigr)\,dx
=2\cH+\int(dP-2F)(\rho)\,dx
$$
with
\begin{equation}\label{eq:F}
\cF:=\IM\int \eps\bar\psi x\cdot\nabla\psi =\int \rho x\cdot
v\,dx\quad\hbox{and }\ P(r):=rf(r)-F(r). 
\end{equation}
Equality \eqref{eq:F} remains formally true if the nonlinearity is not a
pure power.


\subsubsection*{Momentum of inertia or  virial} 

A direct computation shows that the  momentum of inertia  (or virial)
\begin{equation}\label{eq:I}
\cI:=\frac12\int|x|^2|\psi|^2\,dx=\frac12\int|x|^2\rho\,dx
\end{equation}
satisfies  $\displaystyle{\frac {d\cI}{dt}=\cF.}$


\subsubsection*{Pseudo-conformal invariance}

Let $\varphi$ the pseudo-conformal transform of $\psi$ defined by 
$$
\varphi(t,x)=\frac{e^{\frac{i|x|^2}{2\eps t}}}{\(i\frac t\eps\)^{d/2}}
\bar\psi\(\frac{\eps^2}t,\frac xt\)\cdotp
$$
One notices that  
$$
\Bigl(i\eps\d_t\varphi+\frac{\eps^2}2\Delta\varphi\Bigr)(t,x)
=\frac{\eps^2}{t^2}\frac{e^{\frac{i|x|^2}{2\eps t}}}{(i\frac t\eps)^{\frac n2}}
\overline{\Bigl(i\eps\d_t\psi+\frac{\eps^2}2\Delta \psi\Bigr)}
\(\frac{\eps^2}t,\frac xt\).
$$
Thus
$$i\eps\d_t\varphi+\frac{\eps^2}2\Delta\varphi=\frac{\eps^2}{t^2} f\(\(\frac
t\eps\)^d|\varphi|^2\)\varphi. 
$$
For the $L^2$-critical power nonlinearity  $f(r)=r^{2/d},$
one checks that 
$$
\psi(t,x)\ \hbox{ solution of }\  \eqref{NLS_eps} \iff 
\varphi(t,x)    \hbox{ solution of }\  \eqref{NLS_eps}.
$$
For \eqref{QHD_eps}, this yields
$$
(v,\rho)(t,x)\ \hbox{ solution} \iff 
\biggl(\frac xt-\frac vt\Bigl(\frac{\eps^2}t,\frac xt\Bigr),\biggl(\frac\eps t\biggr)^d\rho\Bigl(\frac{\eps^2}t,\frac xt\Bigr)\biggr)
\   \hbox{ solution}.
$$
Using the conservation of energy for  $\varphi,$
one deduces after a lengthy computation that
\begin{equation}\label{eq:Z}
\frac{d\cZ}{dt}+t\int(dP-2F)(|\psi|^2)\,dx=0
\end{equation}
with
$$\begin{array}{lll}
\cZ(t)&\!\!:=\!\!&\Int\Bigl( \frac12|(x+i\eps t\nabla)\psi|^2+t^2F(|\psi|^2)\Bigr)\,dx,\\[1.5ex]
&\!\!=\!\!&\Int\Bigl( \frac12\rho|x-tv|^2+\frac12\eps^2t^2|\nabla\sqrt\rho|^2+t^2F(|\psi|^2)\Bigr)\,dx.
\end{array}
$$
This equality can be proven more simply by using the fact that  
the operators  $i\eps\d_t +\frac{\eps^2}2\Delta$ and $x+i\eps t\nabla$ commute. 
The  conservation law for  $\cZ$  is not independent from the
preceding ones: by expanding the square of the modulus, one observes
that  
$$
\cZ(t)=t^2\cH-t\cF+\cI.
$$
Consequently, 
$$
\frac{d\cZ}{dt}=2t\cH-\cF-t\frac{d\cF}{dt}+\frac{d\cI}{dt}
=2t\cH-t\frac{d\cF}{dt},
$$
and one recovers  \eqref{eq:Z}. 


\subsubsection*{The  Carles-Nakamura conservation law}

Formally, the quantity
$$
\cU(t):=\RE\int\bar \psi\:(x+i\eps t\nabla)\psi\,dx
$$
is constant. 
This can be checked by an indirect way (see \cite{CaNa04})
or by a direct computation. 
The interpretation in terms of \eqref{QHD_eps} is
\begin{equation}\label{eq:U}
\frac d{dt}\cU=0\quad\hbox{with}\quad
\cU(t)=\int\rho(x-tv)\,dx.
\end{equation}
This conservation law also results from the conservation of momentum
and from the conservation law for $\cX$ since in both cases, 
 \eqref{NLS_eps} or  \eqref{QHD_eps}, one has
$$
\cU=\cX-t\cP.
$$


\subsection{Conservation laws for general capillary fluids}
We here study to what extent the conservation laws listed 
in the previous subsection are relevant for general inviscid capillary fluids. 
We recall that such fluids are governed by System
\eqref{eq:capillary}. 
Throughout, we assume  the capillarity $\kappa$ to be  a
differentiable function on  $\R^+.$  
 \smallbreak
For smooth solutions with a non vanishing density, 
System \eqref{eq:capillary} recasts in the following conservative form\footnote{With
the convention  $(\div K)^j:=\sum_i\d_i K_{ij}$}:
 \begin{equation}\label{eq:capillarycons}
\left\{\begin{array}{l}
 \partial_t(\rho v)+\div(\rho v\otimes v)+\nabla P(\rho)=\div K ,\\[1ex]
 \partial_t\rho+\DIV(\rho v)=0,
  \end{array}\right.
 \end{equation}
 with
 $K(\rho,\nabla\rho):=\Bigl(\frac12(\kappa(\rho)+\rho\kappa'(\rho))
|\nabla\rho^2| 
 +\rho\kappa(\rho)\Delta\rho\Bigr)I_d-
\kappa(\rho)\nabla\rho\otimes\nabla\rho.$
\begin{theorem} If $(\rho,v)$ is a sufficiently smooth solution of
  $\eqref{eq:capillary}$ which 
  decays at infinity,  with $\rho$ non
  vanishing, then the equalities $\eqref{eq:M},$  $\eqref{eq:P},$
$\eqref{eq:cA},$ $\eqref{eq:X}$ and  $\eqref{eq:I}$ are still valid, 
and the energy 
$$
\cH:=\int\Bigl(\frac12\rho|v|^2+\frac{\kappa(\rho)}2|\nabla\rho|^2
+F(\rho)\Bigr)\,dx  
$$
is conserved.\medbreak
Furthermore, the quantity $\cF$ defined in  \eqref{eq:F} satisfies
$$
\frac d{dt}\cF=2\cH+\int\Bigl(dP-2F+\frac
d2(\rho\kappa)'|\nabla\rho|^2\Bigr)\,dx. 
$$
\end{theorem}
\begin{proof}
It is very likely that most of those conservation laws could be
derived from Noether theorem. They can also been obtained by
pedestrian computations.  

For $\eqref{eq:M},$ $\eqref{eq:X}$ and $\eqref{eq:I},$ there is no
difference with the  
 QHD since the only equation on  $\rho$ is concerned. 

For  \eqref{eq:P} the proof is obvious from the conservative form
\eqref{eq:capillarycons}. 

In order to prove the conservation of energy, the simplest method is
to take the  $L^2$ scalar product with $\rho v$ of the first equation of \eqref{eq:capillary}.  
The treatment of the terms on the left-hand side of the equation for
$v$ is the same as in the compressible Euler equation, and one obtains
after several integrations by parts and the use of the equation for
$\rho,$ 
$$\begin{array}{lll}
\Frac d{dt}\Int\Bigl(\frac12\rho|v|^2+F(\rho)\Bigr)\,dx
&=&\Int \rho v\cdot\nabla\Bigl(\kappa\Delta\rho
+\Frac12\kappa'|\nabla\rho|^2\Bigr)\,dx,\\[1.5ex]
&=&\Int\d_t\rho\Bigl(\kappa\Delta\rho
+\Frac12\kappa'|\nabla\rho|^2\Bigr)\,dx,\\[1.5ex]
&=&-\Int\kappa\nabla\rho\cdot\d_t\nabla\rho\,dx
-\Frac12\int\kappa'\d_t\rho|\nabla\rho|^2\,dx,\\[1.5ex]
&=&-\Frac d{dt}\Int\Frac\kappa2|\nabla\rho|^2\,dx.\end{array}
$$
To prove \eqref{eq:cA}, one writes (using the Einstein summation convention),
$$
A_i=\int\eps_{ijk}x^j\rho v^k
$$
with $\eps_{ijk}=0$ if two of the indices are equal,
and equal to the  signature of $(i\ j\ k)$ otherwise.
\smallbreak
Using the conservative form of the equation  for  $\rho v$
and integrating by parts,  one thus deduces
$$\begin{array}{lll}
\Frac d{dt}A_i&=&\Int\eps_{ijk}x^j\d_\ell K_{\ell k}\,dx
-\int\eps_{ijk}x^j\d_kP-\int\eps_{ijk}x^j\d_\ell(\rho v^k v^\ell)\,dx,\\[1.5ex]
&=&-\Int\eps_{ijk} K_{j k}\,dx+0
-\int\eps_{ijk}\rho v^jv^k\,dx.\end{array}
$$
The tensors  $K$ and $\rho v\otimes v$ being symmetric, 
the right-hand member of the previous equality vanishes.
\smallbreak
The simplest way to prove the  conservation law on $\cF,$ is to use
the second equation in \eqref{eq:capillarycons}.  
Integrating by parts  the capillary term and treating the other terms
as in the compressible Euler equation one obtains 
$$\begin{array}{lll}
\Frac{d\cF}{dt}&=&\Int x\cdot\div K\,dx
-\int x\cdot\Bigl(\nabla P+\div(\rho v\otimes v)\Bigr)\,dx,\\[1.5ex]
&=&\Int x^j\d_iK_{ij}\,dx+\int (dP+\rho|v|^2)\,dx.
\end{array}
$$
Using the expression of the tensor $K$ and integrating by parts, one gets
$$
\begin{array}{lll}
\Int x^j\d_iK_{ij}\,dx&=&-\Int{\rm tr}\,K\,dx,\\[1.5ex]
&=&-d\Int\Bigl(\frac12(\kappa+\rho\kappa')|\nabla\rho|^2
 +\rho\kappa\Delta\rho\Bigr)\,dx
+\Int\kappa |\nabla\rho|^2,\\[1.5ex]
&=&\Frac d2\Int(\rho\kappa)'|\nabla\rho|^2\,dx
+\Int\kappa |\nabla\rho|^2\,dx,
\end{array}
$$
from which the claimed equality results.
\end{proof}
\begin{remark} The case of  QHD corresponds to  $(\rho\kappa)'=0.$
For capillary fluids, the presence of extra terms in the conservation
law for $\cF$ seems relatively harmless 
if $\rho\mapsto\rho\kappa(\rho)$ is an increasing function. 
One should be able to prove without difficulty that with a pressure
law such that  $dP-2F\ge0,$ one has  
$$
I(t)\Eq t {+\infty} \cH t^2.
$$ 
In the case where  $\rho\mapsto \rho\kappa$ is decreasing, 
we expect the conservation law for $\cF$ to be the key to proving
finite time blow-up results similar to those of \cite{SiderisCMP}. 
\end{remark}


\def\cprime{$'$}
\providecommand{\bysame}{\leavevmode\hbox to3em{\hrulefill}\thinspace}
\providecommand{\MR}{\relax\ifhmode\unskip\space\fi MR }
\providecommand{\MRhref}[2]{%
  \href{http://www.ams.org/mathscinet-getitem?mr=#1}{#2}
}
\providecommand{\href}[2]{#2}


\begin{thebibliography}{10}

\bibitem{ACIHP}
T.~Alazard and R.~Carles, \emph{{WKB} analysis for the {G}ross--{P}itaevskii
  equation with non-trivial boundary conditions at infinity}, Ann. Inst. H.
  Poincar\'e Anal. Non Lin\'eaire \textbf{26} (2009), no.~3, 959--977.

\bibitem{Alinhac}
S.~Alinhac, \emph{Temps de vie des solutions r\'eguli\`eres des \'equations
  d'{E}uler compressibles axisym\'etriques en dimension deux}, Invent. Math.
  \textbf{111} (1993), no.~3, 627--670.

\bibitem{AnMa}
P.~Antonelli and P.~Marcati, \emph{The quantum hydrodynamics system in two
  dimensions}, Arch. Rational Mech. Anal., to appear. Archived as {\tt
  arXiv:1011.4545v1}.

\bibitem{AM09}
\bysame, \emph{On the finite energy weak solutions to a system in quantum fluid
  dynamics}, Comm. Math. Phys. \textbf{287} (2009), 657--686.

\bibitem{Audiard2}
C.~Audiard, \emph{Dispersive smoothing for the {E}uler-{K}orteweg model}, SIAM
  J. Math. Anal., in press.

\bibitem{BJM2}
W.~Bao, S.~Jin, and P.~A. Markowich, \emph{Numerical study of time-splitting
  spectral discretizations of nonlinear {S}chr\"odinger equations in the
  semiclassical regimes}, SIAM J. Sci. Comput. \textbf{25} (2003), no.~1,
  27--64.

\bibitem{BDD07}
S.~Benzoni-Gavage, R.~Danchin, and S.~Descombes, \emph{On the well-posedness
  for the {E}uler-{K}orteweg model in several space dimensions}, Indiana Univ.
  Math. J. \textbf{56} (2007), no.~4, 1499--1579.

\bibitem{BDS08}
F.~B\'ethuel, R.~Danchin, and D.~Smets, \emph{On the linear wave regime of the
  {G}ross-{P}itaevskii equation}, J. Anal. Math. \textbf{110} (2010), 297--338.

\bibitem{BGS2}
F.~B\'ethuel, P.~Gravejat, and J.-C. Saut, \emph{Existence and properties of
  travelling waves for the {G}ross-{P}itaevskii equation}, Stationary and time
  dependent {G}ross-{P}itaevskii equations, Contemp. Math., vol. 473, Amer.
  Math. Soc., Providence, RI, 2008, pp.~55--103.

\bibitem{BGS1}
\bysame, \emph{On the {KP} {I} transonic limit of two-dimensional
  {G}ross-{P}itaevskii travelling waves}, Dyn. Partial Differ. Equ. \textbf{5}
  (2008), no.~3, 241--280.

\bibitem{BGSS1}
F.~B\'ethuel, P.~Gravejat, J.-C. Saut, and D.~Smets, \emph{Orbital stability of
  the black soliton for the {G}ross-{P}itaevskii equation}, Indiana Univ. Math.
  J. \textbf{57} (2008), no.~6, 2611--2642.

\bibitem{BGSS2}
\bysame, \emph{On the {K}orteweg-de {V}ries long-wave approximation of the
  {G}ross-{P}itaevskii equation. {I}}, Int. Math. Res. Not. IMRN (2009),
  no.~14, 2700--2748.

\bibitem{BGSS3}
\bysame, \emph{On the {K}orteweg-de {V}ries long-wave approximation of the
  {G}ross-{P}itaevskii equation {II}}, Comm. Partial Differential Equations
  \textbf{35} (2010), no.~1, 113--164.

\bibitem{BOS}
F.~B\'ethuel, G.~Orlandi, and D.~Smets, \emph{Vortex rings for the
  {G}ross-{P}itaevskii equation}, J. Eur. Math. Soc. (JEMS) \textbf{6} (2004),
  no.~1, 17--94.

\bibitem{BS}
F.~B\'ethuel and J.-C. Saut, \emph{Travelling waves for the
  {G}ross-{P}itaevskii equation. {I}}, Ann. Inst. H. Poincar\'e Phys. Th\'eor.
  \textbf{70} (1999), no.~2, 147--238.

\bibitem{CaBook}
R.~Carles, \emph{Semi-classical analysis for nonlinear {S}chr\"odinger
  equations}, World Scientific Publishing Co. Pte. Ltd., Hackensack, NJ, 2008.

\bibitem{CaNa04}
R.~Carles and Y.~Nakamura, \emph{Nonlinear {S}chr\"odinger equations with
  {S}tark potential}, Hokkaido Math. J. \textbf{33} (2004), no.~3, 719--729.

\bibitem{CazCourant}
T.~Cazenave, \emph{Semilinear {S}chr\"odinger equations}, Courant Lecture Notes
  in Mathematics, vol.~10, New York University Courant Institute of
  Mathematical Sciences, New York, 2003.

\bibitem{ChMa-p}
D.~Chiron and M.~Mari{\c{s}}, \emph{Rarefaction pulses for the nonlinear
  {S}chr{\"o}dinger equation in the transonic limit}, submitted.

\bibitem{ChRo10}
D.~Chiron and F.~Rousset, \emph{The {K}d{V}/{KP}-{I} limit of the nonlinear
  {S}chr\"odinger equation}, SIAM J. Math. Anal. \textbf{42} (2010), no.~1,
  64--96.

\bibitem{CDJLS}
F.~Coquel, G.~Dehais, D.~Jamet, O.~Lebaigue, and N.~Seguin, \emph{Extended
  formulations for {V}an der {W}aals models. {A}pplication to finite volume
  methods}, in preparation.

\bibitem{Danchin02}
R.~Danchin, \emph{Zero {M}ach number limit in critical spaces for compressible
  {N}avier-{S}tokes equations}, Ann. Sci. \'{E}cole Norm. Sup. (4) \textbf{35}
  (2002), no.~1, 27--75.

\bibitem{DGM}
P.~Degond, S.~Gallego, and F.~M{\'e}hats, \emph{An asymptotic preserving scheme
  for the {S}chr\"odinger equation in the semiclassical limit}, C. R. Math.
  Acad. Sci. Paris \textbf{345} (2007), no.~9, 531--536.

\bibitem{Gallo}
C.~Gallo, \emph{Schr\"odinger group on {Z}hidkov spaces}, Adv. Differential
  Equations \textbf{9} (2004), no.~5-6, 509--538.

\bibitem{Gallo08}
\bysame, \emph{The {C}auchy problem for defocusing nonlinear {S}chr\"odinger
  equations with non-vanishing initial data at infinity}, Comm. Partial
  Differential Equations \textbf{33} (2008), no.~4-6, 729--771.

\bibitem{GGZ}
I.~M. Gamba, M.~P. Gualdani, and P.~Zhang, \emph{On the blowing up of solutions
  to quantum hydrodynamic models on bounded domains}, Monatsh. Math.
  \textbf{157} (2009), no.~1, 37--54.

\bibitem{PGX93}
P.~G{\'e}rard, \emph{Remarques sur l'analyse semi-classique de l'\'equation de
  {S}chr\"odinger non lin\'eaire}, S\'eminaire sur les \'Equations aux
  D\'eriv\'ees Partielles, 1992--1993, \'Ecole Polytech., Palaiseau, 1993,
  pp.~Exp.\ No.\ XIII, 13.

\bibitem{PG05}
\bysame, \emph{The {C}auchy problem for the {G}ross-{P}itaevskii equation},
  Ann. Inst. H. Poincar\'e Anal. Non Lin\'eaire \textbf{23} (2006), no.~5,
  765--779.

\bibitem{Ger08}
\bysame, \emph{The {G}ross-{P}itaevskii equation in the energy space},
  Stationary and time dependent {G}ross-{P}itaevskii equations, Contemp. Math.,
  vol. 473, Amer. Math. Soc., Providence, RI, 2008, pp.~129--148.

\bibitem{Glassey}
R.~T. Glassey, \emph{On the blowing up of solutions to the {C}auchy problem for
  nonlinear {S}chr\"odinger equations}, J. Math. Phys. \textbf{18} (1977),
  1794--1797.

\bibitem{Grenier98}
E.~Grenier, \emph{Semiclassical limit of the nonlinear {S}chr\"odinger equation
  in small time}, Proc. Amer. Math. Soc. \textbf{126} (1998), no.~2, 523--530.

\bibitem{Hormander1}
L.~H{\"o}rmander, \emph{The analysis of linear partial differential operators.
  {I}}, second ed., Springer Study Edition, Springer-Verlag, Berlin, 1990,
  Distribution theory and Fourier analysis.

\bibitem{LandauQ}
L.~Landau and E.~Lifschitz, \emph{Physique th\'eorique
  (``{L}andau-{L}ifchitz''). {T}ome {III}: {M}\'ecanique quantique. {T}h\'eorie
  non relativiste}, \'Editions Mir, Moscow, 1967, Deuxi\`eme \'edition, Traduit
  du russe par \'Edouard Gloukhian.

\bibitem{LL}
E.~H. Lieb and M.~Loss, \emph{Analysis}, second ed., Graduate Studies in
  Mathematics, vol.~14, American Mathematical Society, Providence, RI, 2001.

\bibitem{LinZhang}
F.~Lin and P.~Zhang, \emph{{S}emiclassical limit of the {G}ross-{P}itaevskii
  equation in an exterior domain}, Arch. Rational Mech. Anal. \textbf{179}
  (2006), no.~1, 79--107.

\bibitem{LP}
F.~Linares and G.~Ponce, \emph{Introduction to nonlinear dispersive equations},
  Universitext, Springer, New York, 2009.

\bibitem{Madelung}
E.~Madelung, \emph{Quanten theorie in {H}ydrodynamischer {F}orm}, Zeit. F.
  Physik \textbf{40} (1927), 322.

\bibitem{MUK86}
T.~Makino, S.~Ukai, and S.~Kawashima, \emph{Sur la solution \`a support compact
  de l'\'equation d'{E}uler compressible}, Japan J. Appl. Math. \textbf{3}
  (1986), no.~2, 249--257.

\bibitem{Ma-p}
M.~Mari{\c{s}}, \emph{Travelling waves for nonlinear {S}chr\"odinger equations
  with non zero boundary conditions at infinity}, submitted.

\bibitem{GuyCauchy}
G.~M{\'e}tivier, \emph{Remarks on the well-posedness of the nonlinear {C}auchy
  problem}, Geometric analysis of PDE and several complex variables, Contemp.
  Math., vol. 368, Amer. Math. Soc., Providence, RI, 2005, pp.~337--356.

\bibitem{Serre97}
D.~Serre, \emph{Solutions classiques globales des \'equations d'{E}uler pour un
  fluide parfait compressible}, Ann. Inst. Fourier \textbf{47} (1997),
  139--153.

\bibitem{SiderisCMP}
T.~Sideris, \emph{Formation of {S}ingularities in {T}hree-{D}imensional
  {C}ompressible {F}luids}, Comm. Math. Phys. \textbf{101} (1985), 475--485.

\bibitem{Sulem}
C.~Sulem and P.-L. Sulem, \emph{The nonlinear {S}chr\"odinger equation,
  self-focusing and wave collapse}, Springer-Verlag, New York, 1999.

\bibitem{Tao}
T.~Tao, \emph{Nonlinear dispersive equations}, CBMS Regional Conference Series
  in Mathematics, vol. 106, Published for the Conference Board of the
  Mathematical Sciences, Washington, DC, 2006, Local and global analysis.

\bibitem{Taylor3}
M.~Taylor, \emph{Partial differential equations. {III}}, Applied Mathematical
  Sciences, vol. 117, Springer-Verlag, New York, 1997, Nonlinear equations.

\bibitem{ThomannAnalytic}
L.~Thomann, \emph{Instabilities for supercritical {S}chr\"odinger equations in
  analytic manifolds}, J. Differential Equations \textbf{245} (2008), no.~1,
  249--280.

\bibitem{Xin98}
Z.~Xin, \emph{Blowup of smooth solutions of the compressible {N}avier-{S}tokes
  equation with compact density}, Comm. Pure Appl. Math. \textbf{51} (1998),
  229--240.

\bibitem{Z}
V.~E. Zakharov, \emph{Collapse of {L}angmuir waves}, Sov. Phys. JETP
  \textbf{35} (1972), 908--914.

\bibitem{Zhidkov}
P.~E. Zhidkov, \emph{The {C}auchy problem for a nonlinear {S}chr\"odinger
  equation}, JINR Commun., R5-87-373, Dubna (1987), (in Russian).

\bibitem{Zhbook}
\bysame, \emph{Korteweg-de {V}ries and nonlinear {S}chr\"odinger equations:
  qualitative theory}, Lecture Notes in Mathematics, vol. 1756,
  Springer-Verlag, Berlin, 2001.

\end{thebibliography}
\end{document}